\documentclass[11pt]{amsart}

\usepackage{amsmath}
\usepackage[T2A,T1]{fontenc}
\usepackage[russian,english]{babel}

\usepackage{hyperref}
\usepackage{mathtools}
\usepackage{amsthm}
\usepackage{amssymb}
\usepackage{amsfonts}
\usepackage[utf8]{inputenc}
\usepackage{tikz}
\usepackage{tikz-3dplot}
\usepackage{tikz-cd}
\usepackage{float}
\usepackage{verbatim}
\usepackage{enumitem}
\usepackage{booktabs}
\usepackage{comment}
\usetikzlibrary{shapes.geometric, calc}
\usetikzlibrary{arrows.meta}
\usetikzlibrary{hobby}
\usetikzlibrary{patterns}
\usetikzlibrary{patterns.meta}
\usepgflibrary{patterns}
\usetikzlibrary{decorations.markings}
\usetikzlibrary{patterns,patterns.meta}

\numberwithin{equation}{section}
\newtheorem{theorem}{Theorem}[section]

\newtheorem{proposition}[theorem]{Proposition}
\newtheorem{prop}[theorem]{Proposition}
\newtheorem{claim}[theorem]{Claim}
\newtheorem{corollary}[theorem]{Corollary}
\newtheorem{cor}[equation]{Corollary}

\newtheorem{lemma}[theorem]{Lemma}

\theoremstyle{definition}
\newtheorem{definition}[theorem]{Definition}
\theoremstyle{remark}
\newtheorem{remark}[theorem]{Remark}
\newtheorem{example}[theorem]{Example}

\newcommand{\bd}{\partial}

\newcommand{\la}{\langle}
\newcommand{\ra}{\rangle}

\renewcommand{\C}{\mathbb{C}}

\newcommand{\Ind}{\mathrm{Ind}}

\newcommand{\Mtilde}{\tilde{M}}

\DeclareMathOperator{\area}{\mathrm{Area}}

\DeclareMathOperator{\mult}{\mathrm{mult}}

\newcommand{\Id}{\mathrm{Id}}

\newcommand{\Sph}{\mathbb{S}}

\newcommand{\N}{\mathbb{N}}
\newcommand{\Z}{\mathbb{Z}}

\newcommand{\del}{\partial}
\newcommand{\RR}{\mathbb{R}}

\newcommand{\fC}{\mathfrak{C}}
\newcommand{\fU}{\mathfrak{U}}

\DeclareMathOperator{\sff}{\mathrm{I\!I}} 

\newcommand{\Mtetk}{\mathord{\textrm{\foreignlanguage{russian}{Д}}}_k}
\newcommand{\Mtetinfty}{\mathord{\textrm{\foreignlanguage{russian}{Д}}}_\infty}
\newcommand{\MtetR}{\mathord{\textrm{\foreignlanguage{russian}{Д}}}}

\newcommand{\Rcapunder}{\underline{\mathsf{R}}}
\newcommand{\Ecal}{\mathcal{E}}

\newcommand{\R}{\mathbb{R}}

\newcommand{\K}{\mathbb{K}}

\newcommand{\sspan}{\text{span}}

\newcommand{\Ncal}{\mathcal{N}}

\newcommand{\mass}{\bold{M}}

\newcommand{\Area}{\mathrm{Area}}

\newcommand{\Diff}{\mathrm{Diff}}
\newcommand{\Met}{\mathrm{Met}}
\newcommand{\Wcal}{\mathcal{W}}

\newcommand{\RP}{\mathbb{R}P}

\newcommand{\Fix}{\mathrm{Fix}}

\newcommand{\spt}{\mathrm{spt}\,}
\newcommand{\sing}{\mathrm{sing}\,}

\definecolor{light-gray}{gray}{.95}
\definecolor{dark-gray}{gray}{.7}

\begin{document}
\title[Minimal Surfaces in $\mathbb{S}^4$]{Nonorientable minimal surfaces embedded in the round $4$-sphere}
\author[M.~Karpukhin]{Mikhail~Karpukhin}
\author[R.~Kusner]{Robert~Kusner}
\author[P.~McGrath]{Peter~McGrath}
\author[D.~Stern]{Daniel~Stern}
\date{}
\address{Department of Mathematics, University College London, 25 Gordon Street, London, WC1H 0AY, UK} \email{m.karpukhin@ucl.ac.uk}
\address{Department of Mathematics, University of Massachusetts,
Amherst, MA, 01003} \email{profkusner@gmail.com, kusner@umass.edu}
\address{Department of Mathematics, North Carolina State University, Raleigh NC 27695} 
\email{pjmcgrat@ncsu.edu}
\address{Department of Mathematics, Cornell University, Ithaca NY, 14853} \email{daniel.stern@cornell.edu}

\begin{abstract}

We show that every closed, nonorientable surface can be minimally embedded in $\Sph^4$, providing in particular the first known examples of embedded, nonorientable minimal surfaces in $\mathbb{S}^4$ with negative Euler characteristic. All of the surfaces we obtain have area below $8\pi$, which has applications to the existence of nonorientable surfaces minimizing the Willmore functional with prescribed topology in $\R^n$ for $n\geq 4$. Moreover, the number of geometrically distinct embeddings of each nonorientable genus is shown to grow at least exponentially with respect to the genus.  
 
Among these surfaces, we identify a distinguished, highly symmetric family which converges in the large-genus limit to a union of four half-spheres, meeting along a great circle, whose poles are vertices of a regular tetrahedron. 
Rescalings of this family converge to a new singly-periodic nonorientable minimal surface in $\R^4$, which seems to provide the first example of a complete, embedded, nonorientable minimal surface in $\R^4$ without continuous symmetry group.

The proofs further develop the equivariant eigenvalue optimization methodology from our earlier work, applied to carefully chosen families of symmetry groups acting on nonorientable surfaces.  Interestingly, we also find natural pairs of surfaces and group actions for which there is no metric maximizing the first normalized eigenvalue of the Laplacian. 

\end{abstract}

\maketitle
\bibliographystyle{alpha}

\section{Introduction} 

\subsection{Main results} In 1970, Lawson proved \cite{Lawson}  each closed orientable surface can be minimally embedded in the round $3$-sphere $\Sph^3$.  Prior to that work, the equatorial $2$-sphere $\Sph^2$ and the Clifford torus, induced by the map
\begin{align*}
\psi: \R^2 \rightarrow \Sph^3 \subset \C^2,
\quad
\psi(x,y) = \frac{1}{\sqrt{2}}( e^{i x}, e^{iy}),
\end{align*}
were the only known embedded minimal surfaces in $\Sph^3$. 
Lawson's examples are highly symmetric and include an extensively studied family $\xi_{g,1}$ of surfaces of genus $g$ which are conjectured to minimize area and Morse index among minimal surfaces of fixed genus, and Willmore energy among arbitrary surfaces with prescribed topology in $\mathbb{S}^3$ \cite{KusnerWillmore, KapZouIndex}.

Among \emph{nonorientable} 
surfaces,  none can be minimally embedded in $\mathbb{S}^3$, since Alexander duality precludes the existence of any embedded nonorientable surface in $\mathbb{S}^3$. To find a nonorientable analog of Lawson's existence result, one might then relax the embeddedness requirement, and seek minimal immersions of nonorientable surfaces in $\mathbb{S}^3$, or increase the codimension to the lowest where embeddedness is possible, and look for minimal embeddings of nonorientable minimal surfaces in $\Sph^4$. Lawson already addressed the first of these in ~\cite{Lawson}, proving that each non-orientable surface other than $\mathbb{RP}^2$ can be minimally immersed in $\Sph^3$. In the present paper, we resolve the latter problem, proving the following. 
\begin{theorem}
\label{Tmain}
Each closed, nonorientable surface can be minimally embedded in the round $4$-sphere $\Sph^4$, with area $<8\pi$, such that the coordinate functions for the embedding are first eigenfunctions of the Laplacian for the induced metric. 
\end{theorem}

Previously, the state of art on this problem was reminiscent of the status of the orientable problem prior to Lawson's work: until now, the projective plane $\RP^2$ and the Klein bottle $\K$ were the only surfaces known to admit such embeddings, with $\RP^2$ associated to the Veronese map $\psi : \Sph^2 \rightarrow \Sph^4$ given by 
\begin{align*}
\psi(x,y,z) = \sqrt{3}\left(xy, xz, yz, \frac{1}{2} (x^2-y^2), \frac{1}{2\sqrt{3}}(x^2+y^2-2z^2)\right),
\end{align*}
and $\K$ induced by the map $\psi : \R^2 \rightarrow \Sph^4 \subset \C^2 \times \R$ given by
\begin{align*}
\psi(x,y) = \bigg( \frac{1}{\sqrt{2}} e^{i x} \sin y, 
\sqrt{\frac{3}{8}} e^{i2x} \sqrt{ 1+ \frac{1}{3} \sin^2 y},
\sqrt{\frac{5}{8}} \cos y
\bigg)
\end{align*}
(see \cite[Proposition 3]{Penskoi}). The latter surface is often denoted by $\tilde\tau_{3,1}$ and was first constructed by Lawson in~\cite{Lawson}

The upper bound of $8\pi$ for the areas of the surfaces in Theorem \ref{Tmain} has some significance. First, it is the lowest possible area threshold below which one could find minimal surfaces of arbitrary topology: that is, there are no sequences of minimal surfaces in $\mathbb{S}^n$ with unbounded topology and area below $A$ for any $A<8\pi$; see Proposition~\ref{Pvarifold} for a more precise statement. 

Second, we note that minimal surfaces with area below $8\pi$  
have Willmore (bending) energy below $8\pi$, and the existence of such surfaces guarantees the existence of embedded Willmore minimizers in $\mathbb{S}^4$ or $\mathbb{R}^4$ with any prescribed topology, see \cite{West} for a detailed explanation (see also \cite{KusnerWillmore2, BauerKuwert,SimonWillmore}).  This addresses a question asked by Marques-Neves (see \cite[Section 4]{WillmoreSurvey}). Previously, to our knowledge, the only known non-orientable surfaces with Willmore energy below $8\pi$ were the Veronese embedding and $\tilde\tau_{3,1}$ discusssed above. Furthermore, some of the more symmetric examples we construct are plausible candidates for the global Willmore minimizers with prescribed topology in $\mathbb{R}^4$, as we discuss in more detail below.

Next, we note that our construction gives rise to many non-isometric examples.
\begin{theorem}
\label{Tcount}
There is a constant $c>0$ such that for each $\gamma \geq 1$, there exist at least $e^{c\gamma}$ pairwise noncongruent minimal embeddings of the closed surface of nonorientable genus $\gamma$ in $\mathbb{S}^4$. Moreover, these surfaces are embedded by first eigenfunctions of the Laplacian, with area $<8\pi$.
\end{theorem}

\begin{remark}
    While it is widely believed that there exist only finitely many noncongruent minimal embeddings in $\Sph^3$ of each topological type, it seems quite possible that the number of embedded minimal surfaces of a given topological type in $\mathbb{S}^4$ could be infinite. However, we expect minimal embeddings by first eigenfunctions, or with area $<8\pi$, to exhibit more rigidity independently of codimension, and it is possible that the exponential growth rate in Theorem \ref{Tcount} is qualitatively sharp under these constraints.
\end{remark}

We also note that our embeddings realize each given surface as a nonorientable minimal \emph{doubling} of an equatorial $\Sph^2 \subset \Sph^4$ in the sense that the image $M$ meets $\Sph^2$ in a collection of curves and isolated points, and the nearest-point projection $\pi$ is well-defined on $M$ and restricts to a $2$-sheeted covering map on $M \setminus \Sph^2$. 

Among our examples,
we highlight one highly-symmetric family $\Mtetk$. Let $T\subset \R^3$ be a collection of $4$ rays 
emanating from the origin and pointing towards the vertices of a regular tetrahedron, and define $\Mtetinfty:=(\R^2\times T)\cap\Sph^4$ to be a collection of $4$ hemispheres. In Section \ref{Stet}, we find a sequence $\Mtetk$ of embedded non-orientable minimal surfaces with $\chi(\Mtetk) = 4-3k$ converging to $\Mtetinfty$ as $k\to\infty$. The symmetry group of $\Mtetk$ has order $24k$ and combines the action of the tetrahedral group on the $4$ hemispheres with the dihedral action along the common circle; see \ref{Ltetsym} for a precise description.  For these reasons, the family $\Mtetk$ might be considered a non-orientable analog in $\Sph^4$ of the Lawson surfaces $\xi_{k-1,1}$ in $\Sph^3$. 

By taking an appropriate blow-up, we can describe the local model for the surfaces $\Mtetk$ along the singular circle: the limit of this blow-up sequence leads to a new complete minimal surface $\MtetR\subset \R^4$, whose properties we summarize below.  

\begin{theorem}
\label{TR4}
There is a properly embedded minimal surface $\MtetR$ in $\R^4$ which is complete, non-orientable, and singly-periodic.  Furthermore, $\MtetR$ has density two at infinity, and its unique tangent cone at infinity is $\R \times T$.
\end{theorem}

The surface $\MtetR$ appears to be the first example of a complete nonorientable minimal surface embedded in $\R^4$ without continuous symmetry (there are embedded minimal M{\"o}bius bands \cite{Mira, Oliveira, FraserSchoen, FraserSargent, Alarcon} with $\Sph^1$-symmetry). The full symmetries of $\MtetR$ are described in Section \ref{SR4}.  Based on its symmetries and asymptotics, the complete surface $\MtetR$ might be considered a non-orientable analog in $\mathbb{R}^4$ of the most symmetric member of the family of singly-periodic embedded minimal surfaces discovered by Scherk in 1834 \cite{Scherk}.

\subsection{Equivariant eigenvalue optimization.}

Theorems \ref{Tmain} and \ref{Tcount} are proved using a connection between isoperimetric problems in spectral geometry and minimal surfaces in $\Sph^n$ which has emerged in the last thirty years. Following work of Hersch \cite{Hersch}, Berger \cite{Berger}, Yang-Yau \cite{YangYau}, Li-Yau \cite{LiYau}, and others (for example, \cite{MR, ElSoufi}) concerning metrics maximizing the first eigenvalue of the Laplacian, in 1996, Nadirashvili discovered \cite{Nadirashvili} that metrics $g$ on a closed surface which are critical points for the normalized Laplace eigenvalues
\begin{align*}
\bar{\lambda}_1(M, g) = \area(M,g) \lambda_1(M, g)
\end{align*}
are (up to scaling) precisely those induced by branched minimal immersions of $M$ into $\Sph^n$ by first Laplace eigenfunctions.  In particular, \cite{Nadirashvili} shows that $\bar{\lambda}_1$-maximizing metrics must be induced by branched minimal immersions $M \rightarrow \Sph^n$ whose coordinates are first Laplace eigenfunctions. 

Explicit $\bar{\lambda}_1$-maximizing metrics have been found on the projective plane, torus, Klein bottle, and the genus-two orientable surface \cite{LiYau, Nadirashvili,EGJ, Jakobson, NaySh}, with the metrics on the projective plane and Klein bottle induced by the embeddings to $\Sph^4$ mentioned above. The general existence problem for maximizing metrics was recently settled in work of Petrides \cite{PetridesMax} and Karpukhin-Petrides-Stern \cite{KSP}, who showed that each closed surface admits a maximizing metric.  This followed developments in \cite{KKMS}, which also introduced a method to construct minimal surfaces in low-dimensional spheres by maximizing normalized Laplace eigenvalues in the presence of a discrete symmetry group; see ~\cite{Franz} for a recent survey. In \cite{KKMS}, we used this method to find new oriented minimal surfaces embedded in $\Sph^3$, and developments in \cite{KSP} now allow us to extend this method to the nonorientable case. 

As in \cite{KKMS}, the choice of symmetries with respect to which $\bar{\lambda}_1$ is to be maximized plays a key role at every step of the argument, from the existence of maximizing metrics to the structure of the associated minimal surfaces. In \cite{KKMS}, we defined a \emph{basic reflection surface} (BRS) to be any pair $(M,\Gamma)$, where $M$ is a closed surface, and $\Gamma\leq \Diff(M)$ is a finite group of diffeomorphisms containing an involution $\tau\in\Gamma$, whose fixed-point set $M^\tau$ contains a curve, such that the quotient $M/ \langle \tau \rangle$ has genus zero. This gives rise to a large class of pairs $(M,\Gamma)$, and in \cite{KKMS} (see Lemma 5.35 and Corollary 5.9) we showed that any metric maximizing $\bar{\lambda}_1$ among $\Gamma$-invariant metrics on any BRS must be induced by a minimal embedding of $M$ into $\Sph^n$ for some $n$, realizing $M$ as a minimal doubling of an equatorial $\Sph^2 \subset \Sph^n$, with area below $8\pi$. When $M$ is \emph{orientable}, it was also shown moreover in \cite{KKMS} that $n=3$, and $\Gamma$-invariant metrics maximizing $\bar{\lambda}_1$ exist on all simple families of oriented basic reflection surfaces $(M,\Gamma)$.

For nonorientable surfaces $M$, the collection of topologically distinct basic reflection surfaces $(M,\Gamma)$ becomes larger, and neither existence of $\Gamma$-invariant $\bar{\lambda}_1$-maximizing metrics nor the multiplicity bound $n=4$ holds in general. For instance, among the order-two basic reflection surfaces $(M,\langle \tau\rangle)$ generated by a single reflection $\tau$, there is essentially a unique such structure on each orientable $M$, while in the nonorientable setting, these are naturally partitioned into three subcollections, which we call \emph{spherical}, \emph{elliptic}, and \emph{hyperelliptic} depending on the topology of the fixed point set of $\tau$. The spherical class includes surfaces realizing every nonorientable topology, and within this class we can indeed prove both existence of $\bar{\lambda}_1$-maximizing metrics and the multiplicity bound $n=4$, providing one family sufficient for the proof of Theorem~\ref{Tmain}. In the elliptic case, we can show that $\Z_2$-invariant maximizers exist, but are induced by minimal embeddings into $\Sph^6$. Strikingly, in the hyperelliptic case, which includes the majority of order-two BRS on nonorientable surfaces with large Euler characteristic, we find that $\Z_2$-invariant maximizers \emph{never exist}, with the conformal classes of any maximizing sequence diverging to the boundary of the moduli space. This seems to be the first example of an optimization problem for the first eigenvalue $\bar{\lambda}_1$ for which maximizing metrics are shown not to exist, and demonstrates that the conditional existence theorem~\cite[Theorem 8.1]{KKMS} is in a sense sharp; see Remark~\ref{rmk:no_max} for further discussion.

To prove Theorem~\ref{Tcount}, we work with a large class of \emph{order-four} BRS, consisting of nonorientable surfaces endowed with an action of $D_2 = \Z_2 \times \Z_2$ in such a way that each non-trivial element of $D_2$ is a basic reflection. Within this family, we prove that invariant $\bar{\lambda}_1$-maximizers always exist, and multiplicity bounds force the corresponding minimal surface to lie in $\Sph^4$. The pairs $(M,\Gamma)$ in this family have a relatively simple combinatorial description, with each of these $D_2$ actions on the nonorientable surface $M$ with $\chi(M)=4-k$ obtained by gluing 4 copies of a $k$-gon, where $k\geq 3$. More precisely, we start with a $k$-gon whose edges are labeled by the non-trivial elements of $D_2$, with each label appearing at least once, and apply Vinberg's construction (see Section \ref{SD2main}) so that the $k$-gon becomes the fundamental domain of a $D_2$ action such that the edges labeled by $a\in D_2$ are fixed by $a$. Different sets of labels (up to permutation of the elements of $D_2$ and automorphisms of the $k$-gon) then give rise to topologically distinct actions of $D_2$, and it is easy to see that the number of distinct labellings grows exponentially in $k$. To complete the proof of Theorem~\ref{Tcount}, one needs to show that the same metric cannot be invariant under too many different actions of $D_2$; this is achieved in Section \ref{ssCount}, where we identify a set of labels of size $e^{ck}$ giving rise to distinct minimal surfaces.

Furthermore, some of these examples can be endowed with additional symmetries coming from the automorphisms of the $k$-gon. In particular, we identify four highly symmetric families
 called $\eta^{v}_k, \eta^{c}_k, \eta_k^e,$ and $ \Mtetk$. The surfaces $\eta^{v}_k$, $\eta^{c}_k$, and $\eta_k^e$ have the same isometry group, $(D_k\times D_2)\rtimes \Z_2$, as the Lawson surfaces $\xi_{k-1, 1}$ \cite{Lawson}. The family $\Mtetk$ described above is obtained from a regular $3k$-gon labeled by the three elements in the cyclic order $abcabc\ldots abc$ and imposing the maximal symmetry consistent with this labeling. The existence and structure of $\bar{\lambda}_1$-maximal metrics in this family and their behavior in the large topology limit, including the extraction of the entire surface $\MtetR$, is the content of Sections \ref{Stet} and \ref{SR4}. 
 
With the exception of some highly symmetric families like $\Mtetk$, we expect that for most of the families of nonorientable minimal surfaces in $\mathbb{S}^4$ constructed here, the associated varifolds converge in large topology limits to the equatorial $2$-spheres (with multiplicity two) corresponding to the fixed point set of a basic reflection $\tau$. In this case, we identify three plausible models for the geometry of such a minimal surface near the fixed-point set.  Before proceeding further, we note that the fixed-point set of any BRS is a disjoint union of isolated points and embedded loops, called \emph{ovals}, each of which is \emph{twisted} or \emph{untwisted} according to whether a tubular neighborhood is nonorientable or orientable.    

First, the graph $\{ (z, z^2)  \in \C^2 : z \in \C\}$ is a minimal surface in $\C^2$ which is a $\Z_2$-symmetric minimal doubling (in the sense of \ref{ddoub}) of a $2$-plane, with one isolated fixed-point, providing one possible model for the behavior of a basic reflection surface in the vicinity of an isolated fixed-point.  Second, the conformal minimal immersion
\begin{align*}
X: \Sph^1 \times \R \rightarrow \C^2,
\quad
X(t, \theta) = (2 \sinh t \, e^{i \theta}, \cosh 2 t \, e^{i2\theta}) \in \C^2
\end{align*}
of the cylinder $\Sph^1 \times \R$ descends to an embedding of the M{\"o}bius band in $\C^2$; this M{\"o}bius band is a $\Z_2$-symmetric doubling of a $2$-plane in $\C^2$ with a single twisted oval, and provides a plausible model for the geometry in the vicinity of a twisted oval.  Finally, the catenoid in $\R^3$ is a natural model for the geometry in the vicinity of an untwisted oval.  

\subsection{Outline of the paper}  In Section \ref{S:BRS} we collect some definitions and facts we need about basic reflection surfaces.  Section \ref{Smodels}  introduces all of the group actions on nonorientable surfaces that we consider later in the article. In Section \ref{S:Prelim}, we review key results from the existence theory for $\bar{\lambda}_1$-maximizing metrics with symmetry

Sections \ref{SD2main} and \ref{SZ2main} are logically independent of each other and contain the main results with  $D_2$-symmetry and $\Z_2$-symmetry, respectively. Existence and limiting behavior of the minimal surfaces $\Mtetk$ is treated in Section \ref{Stet}, while the extraction of the entire minimal surface $\MtetR$ in $\mathbb{R}^4$ from a blow-up procedure is carried out in Section \ref{SR4}. The remaining highly symmetric families are treated in Sections~\ref{SOther} and~\ref{ssSph2main}. We also include an appendix reviewing several important auxiliary results.

\subsection*{Acknowledgments} The authors would like to thank Alexander West for pointing out the implications of our results for the Willmore existence problem.

M.K. and D.S. would like to thank the Isaac Newton Institute for Mathematical Sciences, Cambridge, for support and hospitality during the programme Geometric spectral theory and applications, where work on this paper was undertaken. This work was supported by EPSRC grant EP/Z000580/1.

M.K. acknowledges support of the European Research Council. Funded by the European Union (ERC Starting Grant EOaMS 101219977). Views and opinions expressed are however those of the authors only and do not necessarily reflect those of the European Union or the European Research Council. Neither the European Union nor the granting authority can be held responsible for them.

P.M.  was partially supported by Simons Foundation Collaboration Grant 838990.  D.S. was partially supported by NSF grant DMS 2404992 and the Simons Foundation.  R.K. was also supported by National Science Foundation grant DMS-1928930 at the Simons Laufer Mathematical Sciences Institute (formerly MSRI) in Berkeley.  R.K. is grateful to NC State University for its hospitality during the Spring
2023 semester, where part of this work was carried out.

\subsection*{AI usage disclosure} No AI was used in the development of the proofs or writing of this article.  
LLM tools, such as Perplexity AI and ChatGPT Free, were used solely for the purposes of literature review. 

\section{Basic Reflection Surfaces}
\label{S:BRS} 

Here we recall facts from \cite{KKMS} regarding closed Basic Reflection Surfaces. The exposition below is condensed, and we refer to \cite{KKMS} for details. Although we are primarily interested in non-orientable surfaces, we need to discuss orientable surfaces as well as they naturally arise as topological degenerations of the latter.

Let $(M,g)$ be a closed Riemannian manifold. For an isometry $\tau$ of $M$, we denote by $M^\tau$ the set of its fixed points. We say that $\tau$ is a {\em reflection} if it is an involution and for some $p\in M^\tau$ the differential $d_p\tau$ is a reflection across a hyperplane.   

\begin{definition}
\label{dref}
If $M$ is a closed, connected Riemannian surface and $\tau$ is a reflection on $M$, the pair $(M, \tau)$ is called a \emph{reflection surface}.  
For a reflection surface $(M,\tau)$, denote by $\pi : M \rightarrow M / \langle \tau \rangle$ the canonical projection. 

A reflection surface $(M,\tau)$ is called {\em basic} (or {\em BRS}\,) if $\pi(M)$ is an orientable surface of genus $0$.
\end{definition}

We now describe the classification of reflection surfaces up to equivariant homeomorphisms (i.e. homeomorphisms preserving the reflections). Recall that by~\cite{Kobayashi}, $M^\tau$ is a disjoint union of totally geodesic submanifolds of $M$, hence, it consists of isolated fixed points and embedded circles. We call the latter an {\em oval} of $\tau$. An oval can be {\em twisted} or {\em untwisted} depending on whether its tubular neighbourhood is non-orientable or orientable respectively. 

\begin{definition}
\label{dspecies}
The \emph{species} associated to a reflection surface $(M, \tau)$ is
\begin{align*}
[\gamma, \epsilon : F, C_+, C_-]
\end{align*} 
where $\epsilon=+$ if $M$ is orientable, $\epsilon = -$ if $M$ is non-orientable; $\gamma$ is the genus of $M$, i.e. $\chi(M) = 2-2\gamma$ if $\varepsilon = +$ and $\chi(M) = 2-\gamma$ if $\varepsilon = -$; $F, C_+, C_-$ are the number of isolated fixed-points, untwisted ovals, and twisted ovals, respectively.  For later use, let $C=C_+ + C_-$ denote the total number of ovals.
\end{definition}

\begin{theorem}[Classification of basic reflection surfaces]
\label{Tclass}

The species
\begin{align*}
[\gamma, \epsilon : F, C_+, C_-]
\end{align*}
of a basic reflection surface $(M, \tau)$ satisfies the following:
\begin{enumerate}[label=\emph{(\roman*)}]
\item $C\geq 1$.
\item If $\epsilon = +$, then $F=C_-=0$ and $C_+ = \gamma+1$. 
\item If $\epsilon = -$, then $F + 2C = \gamma+ 2$ and $F+C_-$ is even and positive.
\end{enumerate}
Moreover, each species
 consistent with (i)-(iii) is realized by a basic reflection surface which is unique up to equivariant homeomorphism.  
\end{theorem}
\begin{proof}
This is a special case of Theorem 5.9 in \cite{KKMS}. 
\end{proof}

\begin{remark}
\label{Rbas}
Basic reflection surfaces realizing each species from Theorem \ref{Tclass} may be constructed by the following surgery procedure:
\begin{enumerate}
\item For $(\gamma, \epsilon) = (\gamma, +)$, double an orientable genus $0$ surface with $\gamma+1$ boundary components along its boundary.  See Figure \ref{Fdoub}.
\begin{figure}[ht]
\centering
\def\ra{4}
\begin{tikzpicture}[scale=.7,transform shape]
\draw [draw=black] (0, 0) ellipse ({\ra-\ra/6} and {\ra/3.5});
\foreach \j in {1, 2}
{
\foreach \i in {-1, 1}
{
\draw [draw=black] ({\i*(\j*\ra/3)-\i*\ra/6}, 0 ) ellipse ({\ra/12 and \ra/20}); 
\draw [draw=black, densely dashed] ({(5*\ra/6}, 0) arc (0:180:{\ra/8} and {\ra/40});
\draw [draw=black] ({(5*\ra/6}, 0) arc (0:-180:{\ra/8} and {\ra/40});
\draw [draw=black, densely dashed] ({\i*(\ra/2)-\i*\ra/6+\ra/12}, 0) arc (0:180:{\ra/12} and {\ra/40});
\draw [draw=black] ({\i*(\ra/2)-\i*\ra/6+\ra/12}, 0) arc (0:-180:{\ra/12} and {\ra/40});
\draw [draw=black, densely dashed] ({(-5*\ra/6+\ra/4}, 0) arc (0:180:{\ra/8} and {\ra/40});
\draw [draw=black] ({(-5*\ra/6+\ra/4}, 0) arc (0:-180:{\ra/8} and {\ra/40});
}
}
\node at (\ra/90, 0) {$\boldsymbol{\cdots}$}; 
\end{tikzpicture}  
        \caption{A basic reflection surface with topological type $(\gamma, +)$.}
        \label{Fdoub}
\end{figure}
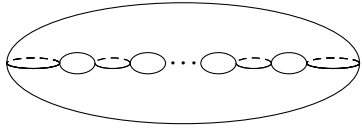
\item For $(\gamma, \epsilon) = (\gamma, -)$, start with a genus $n: = (\gamma-C_- -2C_+)/2$ surface with an involution $\tau$ fixing $2n+2$ isolated points. Replace neighborhoods of $C_-$ of these points with crosscaps and extend $\tau$ in a natural way.  Finally, remove $C_+$ pairs of $\tau$-invariant disks, and identify the boundaries of these disks via $\tau$.  See Figure \ref{Fhandlebody}.
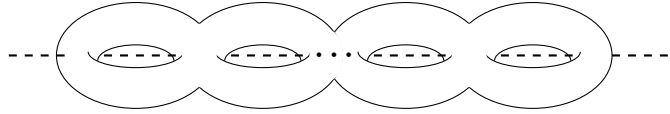
\begin{figure}[ht]
\begin{tikzpicture}[scale=.7]
                \begin{scope}[xscale=.6]
         \draw[thick, dashed] (-4, 0)--(-2, 0);
         \draw[thick, dashed] (-1, 0)--(2, 0);
         \draw[thick, dashed] (3, 0)--(6, 0);
         \draw[thick, dashed] (7.5, 0)--(10, 0);
         \draw[thick, dashed] (11.5, 0)--(14, 0);
         \draw[thick, dashed] (15, 0)--(17, 0);
          \draw (0,0) ellipse (2.5cm and 1cm);
          \draw (-1.5,.07) arc (180:360:1.5cm and .3cm);
          \draw (1.2,-.1) arc (0:180:1.2cm and .3cm);
          \draw (4,0) ellipse (2.5cm and 1cm);
          \draw (2.5,.07) arc (180:360:1.5cm and .3cm);
          \draw (5.2,-.1) arc (0:180:1.2cm and .3cm);
          \fill[color=white](2,.6) to [bend right=50] (1.45,0) to [bend right = 50] (2,-.6) to [bend right=50] (2.55,.0) to [bend right=50] (2,.6);
          \begin{scope}[xshift=8.5cm]
          \draw (0,0) ellipse (2.5cm and 1cm);
          \draw (-1.5,.07) arc (180:360:1.5cm and .3cm);
          \draw (1.2,-.1) arc (0:180:1.2cm and .3cm);
          \draw (4,0) ellipse (2.5cm and 1cm);
          \draw (2.5,.07) arc (180:360:1.5cm and .3cm);
          \draw (5.2,-.1) arc (0:180:1.2cm and .3cm);
          \fill[color=white](2,.6) to [bend right=50] (1.45,0) to [bend right = 50] (2,-.6) to [bend right=50] (2.55,.0) to [bend right=50] (2,.6);
          \fill[color=white](-2.5,0) ellipse (.53cm and .5cm);
          \end{scope}
          \end{scope}
          \node at (3.8, 0) {$\boldsymbol{\cdots}$}; 
          \end{tikzpicture}
          \caption{An involution $\tau$ on a genus $n$ surface fixing $2n+2$ points.  As depicted, $\tau$ is the half-turn about the dashed line.}
                   \label{Fhandlebody}
          \end{figure}
\end{enumerate}
\end{remark}

\begin{definition}
\label{Nelliptic}
Let $(M,\tau)$ be a nonorientable basic reflection surface, and let $g$ be such that $2g+2 = F+C_-$.  The pair $(M, \tau)$ is called
\begin{itemize}
\item \emph{spherical} if $g = 0$;
\item \emph{elliptic} if $g = 1$; and
\item \emph{hyperelliptic} if $g \geq 2$.
\end{itemize}
\end{definition}
Notice that $g$ in Definition \ref{Nelliptic} is the genus of the orientable surface associated to $M$ via the surgery description of $M$ in Remark \ref{Rbas}(2).

\section{Models for nonorientable minimal surfaces}
\label{Smodels}
In this section, we describe some natural group actions $T: \Gamma \times M \rightarrow M$ on nonorientable closed surfaces $M$ with the property that $(M, \rho)$ is a basic reflection surface for at least one element $\rho \in \Gamma$. Throughout, let $D_k$ denote the dihedral group with $2k$ elements. 

\subsection{$\Z_2$-symmetry} 
\label{ssZ2}
Let $M$ be a nonorientable surface and $T: \Gamma \times M \rightarrow M$ be an action of the group $\Gamma = \Z_2 = \langle \tau \rangle $ such that $\tau$ is a basic reflection.  By Remark \ref{Rbas}(2), any such pair $(M, T)$ is constructed by applying an equivariant surgery procedure to a genus-$g$ orientable surface ($g \geq 0$) equipped with an involution fixing $2g+2$ points.  Recalling Definition \ref{Nelliptic}, the pair $(M, T)$ is called spherical, elliptic, or hyperelliptic corresponding to the cases where $g$ is zero, one, or bigger than one.

Theorems \ref{Tmain}, \ref{Telliptic}, and \ref{Tmainhigh} address the problem of whether the supremum $\Lambda_1^T(M)$ is realized by a smooth, $T$-invariant $\bar{\lambda}_1$-maximizing metric, and the following trichotomy holds: 
\begin{itemize}
\item If $(M,T)$ is spherical, $\Lambda^T_1(M)$ is realized by a smooth metric, induced by a minimal, $T$-invariant embedding of $M$ into $\Sph^4$.
\item If $(M, T)$ is elliptic, $\Lambda^T_1(M)$ is realized by a smooth metric, induced by a minimal, $T$-invariant embedding of $M$ into $\Sph^6$.
\item If $(M, T)$ is hyperelliptic, $\Lambda^T_1(M)$ is not realized by a smooth metric.
\end{itemize}

\subsection{$D_k\times \Z_2$-symmetry}
\label{ssSph2}
Let $(M, \tau)$ be a spherical basic reflection surface as in \ref{ssZ2}, so that $M$ is constructed from applying equivariant surgeries to a sphere $S^2$ equipped with an involution $\tau$ fixing two poles. Observe that $S^2$ admits an action of $D_k$ preserving the same two poles. If the set of pairs of centers $(p_i,\tau(p_i))$ of identified disks is $D_k$-invariant actions, then $\Z_2$-action in \ref{ssZ2} extends to a $D_k \times \Z_2$-action $T$ on $M$. Furthermore, a similar construction works with a slightly larger group $D_k\times \Z_2$, where $\Z_2$ acts on $S^2$ by the reflection interchanging the poles.

\begin{example}[The families $\zeta_k^i$]
\label{ex:zetaki}
The simplest families of surfaces with a $D_{2k} \times \Z_2$ action are obtained by placing $2k$ centers of disks along the equator of $S^2$ preserved by the $D_k$ action, and gluing each pair of boundary circles which are exchanged by the involution. We call this surface $\zeta_k^0$; it has $k$ untwisted ovals, zero twisted ovals, and two isolated fixed-points. Furthermore, one can attach a cross-cap at either of the $2$ fixed points of $S^2$, we call these surfaces $\zeta_k^i$, where $i$ is the number of attached cross-caps. If $i$ is even, then we endow $\zeta_k^i$ with an additional reflection interchanging the poles leading to the $D_{2k}\times\Z_2\times\Z_2$ symmetry group. The Euler characteristic of $\zeta_k^i$ is $2-2k-i$, which covers all possible non-orientable topologies.
\end{example}

Theorem \ref{TsphDk} shows each $(M, T)$ arising this way has a smooth $T$-invariant metric realizing the supremum $\Lambda^T_1(M)$, induced by a minimal $T$-invariant embedding of $M$ into $\Sph^4$.

\subsection{$D_2$-symmetry}
\label{ssD2}

Let us first briefly recall Vinberg's construction---a convenient way to generate manifolds with actions of groups by reflections, see also Section~\ref{ssD2chamber}. Let $\Omega$ be a surface with boundary, $\Gamma$ be a finite group and $\gamma_i$ be a set of order $2$ generators of $\Gamma$. For each $i$, let $\Omega^{\gamma_i}\subset \bd \Omega$ be a collection of closed segments in $\bd\Omega$ such that $\bd\Omega = \cup_i\Omega^{\gamma_i}$ and interiors of $\Omega^{\gamma_i}$ in $\bd\Omega$ are disjoint. It is convenient to depict the boundary $\bd\Omega$ as a piecewise smooth curve with edges corresponding to connected components of $\Omega^{\gamma_i}$, in which case we refer to $\gamma_i$ as a ``label'' of the corresponding edge. A surface $\Omega$ with those labels is referred to as {\em chamber}, while the edges are called {\em walls} or {\em panels}.
For each $x\in\bd\Omega$ define $\Gamma_x$ to be the subgroup of $\Gamma$ generated by $\gamma_i$ such that $x\in\Omega^{\gamma_i}$. One can define $M(\Gamma,\Omega)$ as a quotient space of $\Gamma\times \Omega$ by the equivalence  relation $\sim$ defined by $(g, x) \sim (h, y) \iff x = y$ and $g^{-1} h \in \Gamma_x$. Then $M$ is a closed surface with the action of $\Gamma$ such that $\Omega$ can be identified with its fundamental domain in such a way that the stabilizer of $x\in\bd\Omega$ is exactly $\Gamma_x$. 

\begin{example}
Let $\Gamma = D_2=\Z_2 \times \Z_2$, and let $\{\rho_1, \rho_2\}$ be a set of generators. Let $\Omega$ be a disk with boundary seprated into $2g+2$ edges labeled by $\rho_1$ and $\rho_2$ in an alternating fashion, see Figure \ref{fig:chamber_or} on the left. Then $N_g=M(D_2,\Omega)$ is an orientable genus $g$ surface, where $\rho_1,\rho_2$ are basic reflections and $\rho_3 = \rho_1\rho_2$ is an involution with $2g+2$ fixed points, as is $\tau$ on Figure~\ref{Fhandlebody}.
\end{example}

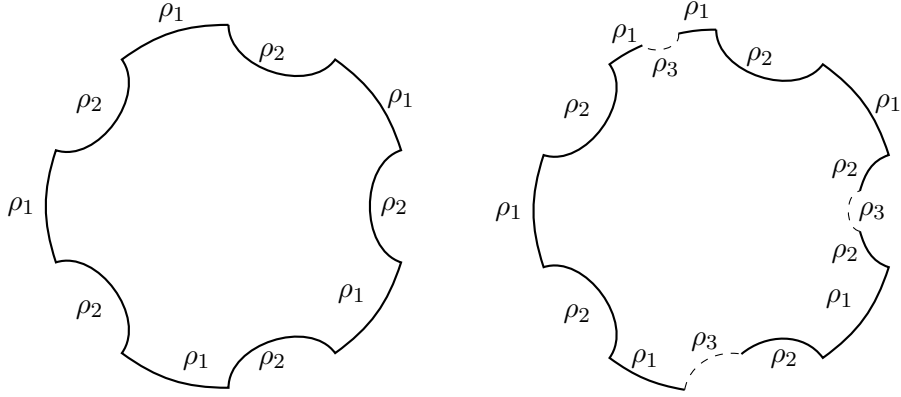
\begin{figure}[h]
\begin{minipage}{.48\textwidth}
\begin{tikzpicture}[scale=1.2]
  \def\R{2}
  \foreach \k in {0,...,9} {
    \coordinate (P\k) at ({\R*cos(90+36*\k)},{\R*sin(90+36*\k)});
  }
  \draw[thick]
    (P0) to[out=180,in=36] node[midway, above] {$\rho_1$} (P1)
    to[out=306, in=-18] node[midway, above left] {$\rho_2$} (P2)
    to[out=252, in=108] node[midway, left] {$\rho_1$} (P3)
    to[out=378, in=54] node[midway, below left] {$\rho_2$} (P4)
    to[out=324, in=180] node[midway, above right] {$\rho_1$} (P5)
    to[out=90, in=126] node[midway, below] {$\rho_2$} (P6)
    to[out=36, in=252] node[midway, above left] {$\rho_1$} (P7)
    to[out=162, in=198] node[midway, right] {$\rho_2$} (P8)
    to[out=108, in=324] node[midway, right] {$\rho_1$} (P9)
    to[out=234, in=270] node[midway, above] {$\rho_2$} (P0);
\end{tikzpicture}
\end{minipage}
\hfill
\begin{minipage}{.48\textwidth}
\begin{tikzpicture}[scale=1.2]
  \def\R{2}
    \def\a{10}
    \def\r{1.6}
  \foreach \k in {0,...,9} {
    \coordinate (P\k) at ({\R*cos(90+36*\k)},{\R*sin(90+36*\k)});
  }
  \foreach \k in {1,...,2} {
    \coordinate (Q\k) at ({\R*cos(90+12*\k)},{\R*sin(90+12*\k)});
  }

  \coordinate (R1) at ({\R*cos(-90 -\a)},{\R*sin(-90 - \a)});
  \coordinate (R2) at ({\r*cos(-90 +\a)},{\r*sin(-90 + \a)});

  \coordinate (S1) at ({\r*cos(90+36*7+10)},{\r*sin(90+36*7+10)});
  \coordinate (S2) at ({\r*cos(90+36*8-10)},{\r*sin(90+36*8-10)});

  \draw[thick]
    (P0) to[out=180,in=12] node[midway, above] {$\rho_1$} (Q1)
    (Q2) to[out=204, in=36] node[midway, above] {$\rho_1$} (P1)
    to[out=306, in=-18] node[midway, above left] {$\rho_2$} (P2)
    to[out=252, in=108] node[midway, left] {$\rho_1$} (P3)
    to[out=378, in=54] node[midway, below left] {$\rho_2$} (P4)
    to[out=324, in = 170] node[midway, above] {$\rho_1$} (R1) 
    (R2) to[out = 40,in=126] node[midway, below] {$\rho_2$} (P6)    
    to[out=36, in=252] node[midway, above left] {$\rho_1$} (P7)
    to[out = 162,in=-74] node[midway, left] {$\rho_2$}(S1)   
    (S2) to[out = 74, in = 198] node[midway, left] {$\rho_2$} (P8)
     to[out=108, in=324] node[midway, right] {$\rho_1$} (P9)
    to[out=234, in=270] node[midway, above] {$\rho_2$} (P0);
    \draw[dashed]
    (Q1) to[out=282, in=-42] node[midway, below] {$\rho_3$} (Q2)
    (R1)  to[out=80, in= 170] node[midway, above] {$\rho_3$} (R2)
    (S1) to[out = 170, in =-170] node[midway, right] {$\rho_3$} (S2);
\end{tikzpicture}
\end{minipage}
    \caption{Depictions of (left) a chamber $\Omega$ for an orientable surface $N$ with a pair of basic reflections $\rho_1$ and $\rho_2$, and  (right) a chamber for a non-orientable surface with a $D_2$-action resulting from $\Omega$ by applying surgeries as in \ref{ssD2}.}
    \label{fig:chamber_or}
\end{figure}

\begin{figure}
    \centering
\begin{tikzpicture}
    \def\h{1}
    
  \draw[draw opacity = 0, pattern={Lines[angle=45,distance=6pt,line width=0.4pt]},
  pattern color=black]

    (0,0) 
    arc (180:360:1cm and 0.3cm) 
    -- (2,\h)
    arc (0:-180:1cm and 0.3cm) 
    -- (0,0);

    \draw[very thick]
    (0,0) -- node[midway,left]{$\rho_1$} (0,{2*\h})
    (2,0) -- node[midway,right]{$\rho_1$} (2,{2*\h})
    (0,\h) arc (180:360:1cm and 0.3cm) node[midway,above]{$\rho_3$};

    \draw
    (0,0) arc (180:360:1cm and 0.3cm)
    (0,2*\h) arc (180:360:1cm and 0.3cm)
    (2,2*\h) arc (0:180:1cm and 0.3cm);

    \draw[dashed]
    (2,0) arc (0:180:1cm and 0.3cm);

    \draw[very thick,dashed]
    (2,\h) arc (0:180:1cm and 0.3cm);

\end{tikzpicture}
    \begin{tikzpicture}[scale=.8]
        \def\w{3.5}

    \draw[draw opacity = 0, pattern={Lines[angle=45,distance=6pt,line width=0.4pt]}, pattern color=black]

    (0,0) -- (\w,0) -- (\w,1) -- (0,1)
    -- (0,0);

    \draw[very thick, postaction={decorate},
    decoration={markings, mark=at position 0.7 with {\arrow{Stealth}}}]
    (0,0) -- node[midway,left]{$\rho_1$} (0,2);
    \draw[very thick, postaction={decorate},
    decoration={markings, mark=at position 0.7 with {\arrow{Stealth}}}]
    ({2*\w},2) -- node[midway,right]{$\rho_1$} ({2*\w},0);
    \draw[very thick]
    (0,1) --node[midway,above]{$\rho_3$} (\w,1) --node[midway,above]{$\rho_3$} ({2*\w},1)
    (\w,0) --node[midway,right]{$\rho_2$} (\w,1) --node[midway,right]{$\rho_2$} (\w,2);
    \draw
    (0,0) -- ({2*\w},0)
    (0,2) -- ({2*\w},2);
    \draw[white](0, -.5)--({2*\w}, 0); 
    \end{tikzpicture}
    
    \caption{The group action on the cylinder (left) attached on the edge labeled $\rho_1$ and the group action on the M\"obius band (right) attached at the vertex. In each case, the thick lines are fixed point sets and the dashed region is the new region in the chamber.}
    \label{fig:surg1}
\end{figure}
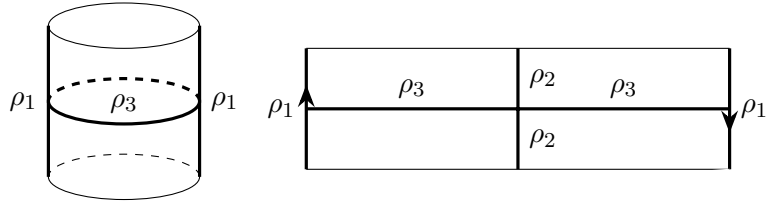

A nonorientable surface $M$ with a $D_2$-action $T$  such that $\rho_3$ is a basic reflection can then be constructed by applying the following $D_2$-equivariant surgeries to $N_g$. One surgery is to let $p\in \mathrm{int}(N_g^{\rho_1})$, cut out small disks around $p$ and $\rho_2(p) = \rho_3(p)$ and identify the boundary points using $\rho_3$. Equivalently, one can attach a cylinder at those circles, see Figure~\ref{fig:surg1} on the left. The other surgery is to let $p\in N^{\rho_1}\cap N^{\rho_2}$ be a point fixed by the $D_2$-action, cut out a small disk around $p$ and glue together the boundary points using $\rho_3$. Equivalently, one can attach a M\"obius band at the boundary circle, see Figure~\ref{fig:surg1} on the right. The corresponding chamber for the action on any such pair $(M, T)$ can be obtained by applying the following operation to the chamber $\Omega$ for $N_g$
\begin{itemize}
\item Replace an isolated fixed-point for $\rho_3$ with a new edge labeled $\rho_3$.
\item Replace a segment of a $\rho_1$- or $\rho_2$-edge with a new edge with label $\rho_3$.
\end{itemize}  
Note that each isolated fixed-point for $\rho_3$ is a vertex of $\Omega$, so the first operation amounts to replacing such a vertex with a $\rho_3$-labeled edge, and also that one could add several $\rho_3$-edges along any edge. Figure~\ref{fig:chamber_or} depicts an example with one of each type of surgery.

The same set of non-orientable surfaces with $D_2$-action can be obtained via Vinberg's construction starting from a polygon labeled by any string  of three letters $\rho_1$, $\rho_2$, $\rho_3$ as long as all labels are present.  

Theorem \ref{TD2} shows that any such pair $(M, T)$ admits a smooth $T$-invariant metric realizing $\Lambda^T_1(M)$, induced by a minimal $T$-invariant embedding of $M$ into $\Sph^4$.  

\subsection{$D_k\times D_2$-symmetry}
\label{ssDkD2}

Let $g \in \N$ and $N = N_g$ be an orientable genus-$g$ surface with a $D_2$-action as in \ref{ssD2}.  Let $k = g+1$.  Then $N$ admits an action of $\Gamma = D_k \times D_2$, arising from the label-preserving action of $D_k$ on the $2k$-gon $\Omega$. The chamber for the $D_k\times D_2$ action is a quadrilateral bounded by two edges fixed by reflections $\gamma_1, \gamma_2$ generating $D_k$, followed by two edges of $\Omega$ fixed by $\rho_1$ and $\rho_2$. 
Making surgeries as in \ref{ssD2} in a way that preserves this action leads to actions $T$ of $\Gamma$ on 
the resulting nonorientable surfaces $M$.

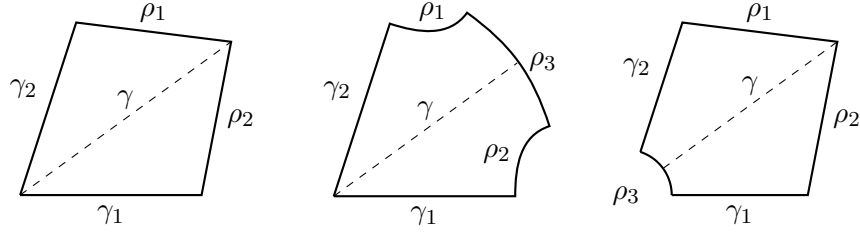
\begin{figure}[ht]
\begin{minipage}{.3\textwidth}
\begin{tikzpicture}[scale=1.5]
  \def\R{2}
    \def\a{10}
    \def\r{1.6}
    \def\s{2.3}
    \coordinate (O) at (0,0);
    \coordinate (P) at ({\r*cos(0)},{\r*sin(0)});
    \coordinate (C) at ({\s*cos(36)}, {\s*sin(36)});
    \coordinate (S) at ({\r*cos(72)},{\r*sin(72)});
    \draw[thick]
    (O) -- node[midway, below] {$\gamma_1$} (P)
    to node[midway, right] {$\rho_2$}  (C)to  node[midway, above] {$\rho_1$} (S) 
    -- node[midway,above left] {$\gamma_2$} (O);
     \draw[dashed]
    (O) -- node[midway, above] {$\gamma$} (C);
\end{tikzpicture}
\end{minipage}
\begin{minipage}{.3\textwidth}
    \centering
\begin{tikzpicture}[scale=1.5]
  \def\R{2}
    \def\a{10}
    \def\r{1.6}
    \coordinate (O) at (0,0);
    \coordinate (P) at ({\r*cos(0)},{\r*sin(0)});
    \coordinate (Q) at ({\R*cos(18)},{\R*sin(18)});
    \coordinate (R) at ({\R*cos(54)},{\R*sin(54)});
    \coordinate (S) at ({\r*cos(72)},{\r*sin(72)});
    \coordinate (CC) at ({\R*cos(36)}, {\R*sin(36)});
    \draw[thick]
    (O) -- node[midway, below] {$\gamma_1$} (P)
    to[out=90,in = -162]  node[midway, left] {$\rho_2$} (Q)
    to[out = 108, in = -36] node[midway,right] {$\rho_3$} (R)
    to[out=-126, in = -18] node[midway,above] {$\rho_1$} (S)
    -- node[midway,above left] {$\gamma_2$} (O);
        \draw[dashed]
    (O) -- node[midway, above] {$\gamma$} (CC);
\end{tikzpicture}
\end{minipage}
\begin{minipage}{.3\textwidth}
\centering
\begin{tikzpicture}[scale=1.5]
  \def\R{2}
    \def\a{10}
    \def\r{1.6}
    \def\s{2.3}
    \def\t{.4}
    \coordinate (O) at (0,0);
    \coordinate (P) at ({\r*cos(0)},{\r*sin(0)});
    \coordinate (C) at ({\s*cos(36)}, {\s*sin(36)});
    \coordinate (S) at ({\r*cos(72)},{\r*sin(72)});
    \draw[thick, domain = 0: 72] plot ( {\t*cos(\x)}, {\t*sin(\x)});
    \draw node[]{$\rho_3$} (0, 0); 
    \draw[thick] ({\t}, {0}) -- node[midway, below] {$\gamma_1$} (P)
    to node[midway, right] {$\rho_2$}  (C)to  node[midway, above] {$\rho_1$} (S) 
    -- node[midway,above left] {$\gamma_2$} ({\t*cos(72)}, {\t*sin(72)});
     \draw[dashed]
   ({\t*cos(36)}, {\t*sin(36)}) -- node[midway, above] {$\gamma$} (C);
\end{tikzpicture}
\end{minipage}
    \caption{Depictions of chambers for $D_k \times D_2$ actions on (left) the orientable surface $N_g$, (middle) a non-orientable surface obtained by replacing the $\rho_1\rho_2$-vertex with a $\rho_3$-labeled edge, and (right) a nonorientable surface obtained by replacing the $\gamma_1\gamma_2$-vertex with a $\rho_3$-labelled edge.  In each case, either side of the bisector $\gamma$ is a chamber for the action of $\Gamma_k = (D_k \times D_2) \rtimes \Z_2$.} 
    \label{fig:chamber_Dk_nor}
\end{figure}

Note that the (unlabeled) chamber on the left of Figure~\ref{fig:chamber_Dk_nor} is symmetric with respect to the reflection $\gamma$ about the bisector of the angle between the $\gamma_1$ and $\gamma_2$ edges. However, the labels are not preserved by $\gamma$, which means that this action cannot be extended to the action of $\Z_2\times D_k\times D_2$. However, since $x\in N_g^s$ iff $\gamma(x)\in N_g^{\gamma\tau\gamma^{-1}}$, the action can be extended to the semidirect product $(D_k\times D_2)\rtimes\Z_2$, where the action of $\gamma$ by conjugations on generators of $D_k\times D_2$ is given by 
\begin{align}
\label{EconjDkD2}
\gamma\gamma_2\gamma = \gamma_1
\quad
\text{and}
\quad
\gamma\rho_1\gamma = \rho_2,
\end{align}
and then necessarily $\gamma\rho_3 = \rho_3\gamma$. For each $k > 2$, this is exactly the description of the action of the full isometry group of the Lawson surface $\xi_{k-1, 1}$. To obtain non-orientable surfaces with such an action, one performs equivariant surgeries on the triangular chamber, which is a half of the picture on the left of  Figure~\ref{fig:chamber_Dk_nor}.

\begin{example}[The families $\eta^{v}_k$, $\eta^{c}_k$, and $\eta^e_k$]
\label{ExDkD2}
We call the simplest families of surfaces arising in this way $\eta^{v}_k$, $\eta^{c}_k$, and $\eta^e_k$.  Fundamental chambers for $\eta^v_k$ and $\eta^c_k$ are depicted  in the second and third images in Figure~\ref{fig:chamber_Dk_nor}, and for each of these surfaces, it is straightforward to check that $\rho_3$ provides a basic reflection. The surface $\eta^{v}_k$ has $2k$ twisted ovals for $\rho_3$, arising from replacing all of the isolated fixed-points of $N_g$ with twisted ovals, and no isolated fixed-points or untwisted ovals.  The surface $\eta^{c}_k$ has $2k$ isolated fixed-points, two untwisted ovals and no twisted ovals for $\rho_3$, resulting from replacing the $\gamma_1\gamma_2$ vertex in the chamber with an edge labelled $\rho_3$.  Finally, $\eta^e_k$ has $2k$ untwisted ovals, arising from adding a $\rho_3$ edge to the middle of each $\rho_1$ and $\rho_2$ edge, $2k$ isolated fixed-points, and no twisted ovals.  (In the context of the first image in Figure \ref{fig:chamber_Dk_nor} this amounts replacing the $\gamma_2\rho_1$ vertex with an edge labeled $\rho_3$.)

The surface $\eta^{v}_k$ has Euler characteristic $4-4k$ and nonorientable genus $4k - 2$, while $\eta^{c}_k$ has Euler characteristic $-2k$ and nonorientable genus $2k+2$, and $\eta^e_k$ has Euler characteristic $4-6k$ and nonorientable genus $6k-2$.
\end{example}

\begin{remark}
If $M$ is one of $\eta^v_k, \eta^e_k$, or any of the examples discussed in \ref{ssD2}, then each of the reflections $\rho_1,\rho_2, \rho_3$ is a basic reflections (see Lemma \ref{LD2quo}(iv)) on $M$.  On the other hand, for $\eta^c_k$, the reflection $\rho_3$ is the only one among $\rho_1,\rho_2,\rho_3$ which is basic.  In general, one could consider other families of $D_2$-symmetric (or $D_k\times D_2$-symmetric) BRS for which $\rho_1,\rho_2$ are \emph{not} basic, obtained by surgeries which remove disks in the interior of the chamber $\Omega$ and identify the boundary points using $\rho_3$.  For simplicity, $\eta^c_k$ is the only surface we discuss arising in this way.
\end{remark}

\subsection{The family $\Mtetk$} 
\label{ssD3}

Here is another highly-symmetric family of examples.  For any $k\geq 1$, let $\Mtetk$ be the nonorientable surface with a $D_2$-action arising from surgeries as in \ref{ssD2} which replace every other vertex on the $2k$-gon fundamental chamber with $\rho_3$-labelled edges.  Then $\Mtetk$ has $k$ isolated fixed-points,  $k$ twisted ovals, and no untwisted ovals.  In particular, $\Mtetk$ has Euler characteristic $4-3k$ and nonorientable genus $3k-2$.

The surface $\Mtetk$ can be naturally endowed with a larger group action as follows.  The fundamental chamber for $\Mtetk$ is a $3k$-gon; consider the dihedral group $D_{3k}$, and note that the subgroup $\Z_k$ generated by rotations by $2\pi/k$ is normal. 
Then one has $D_{3k}/\Z_k\cong D_3\cong S_3\cong \mathrm{Aut}(D_2)$, where the last isomorphism follows from the fact that any permutation of the nonidentity elements of $D_2$ is an automorphism of $D_2$. Therefore, one can form $\Gamma_k = D_2 \rtimes_\phi D_{3k}$, where $\phi\colon D_{3k}\to \mathrm{Aut}(D_2)$ is the map described above. Let us now explain how to endow $\Mtetk$ with an action of $\Gamma_k$. 

Notice that the edges of the fundamental chamber for the $D_2$ action have labels $\rho_1, \rho_2, \rho_3$ repeating $k$ times in counterclockwise order, so that the labels are invariant under the action of $\Z_k$, which is consistent with $\ker\phi = \Z_k$. Let $\gamma_1$ and $\gamma_2$ be reflections generating $D_{3k}$, whose images under $\phi$ are also reflections generating the factor $D_3$; to see that $\Gamma_k$ acts on $\Mtetk$ in the desired way, it is sufficient to check that $\gamma_1$ and $\gamma_2$ act on labels of the chamber via the automorphism $\phi$. Let $w_1$ be the word $\rho_1\rho_2\rho_3$ and $w_1^{-1} = \rho_3\rho_2\rho_1$ be the same word read from right to left. Observe that $\gamma_1$ transforms $w_1$ to $w_1^{-1}$, which agrees with the action of an automorphism of $D_2$ that exchanges $\rho_3$ and $\rho_1$ and preserves everything else. A similar argument for $\gamma_2$ with $w_2 = \rho_2\rho_3\rho_1$ shows that the action of $\gamma_2$ is consistent with the action by the element of $\mathrm{Aut}(D_2)$ that exchanges $\rho_3$ and $\rho_2$. 

For the reader's convenience, properties of some of the surfaces discussed in this section are summarized in the following table.

\begin{table}[ht]
\caption{Families of nonorientable minimal surfaces in $\Sph^4$}
\label{table1}
\begin{tabular}{@{}|c|c|c|c|c|c|@{}}
\toprule
Surface                   & Genus  & Symmetry Group                & $C_+$ & $C_-$ & $F$  \\ \midrule
$\Mtetk$ & $3k-2$ & $D_2 \rtimes D_{3k} $          & $0$   & $k$   & $k$  \\ \midrule
$\eta^{v}_k$         & $4k-2$ & $(D_k\times D_2)\rtimes \Z_2$ & $0$   & $2k$  & $0$  \\ \midrule
$\eta^{c}_k$        & $2k+2$ & $(D_k\times D_2)\rtimes \Z_2$ & $2$   & $0$   & $2k$ \\ \midrule
$\eta^{e}_k$        & $6k+2$ & $(D_k\times D_2)\rtimes \Z_2$ & $0$   & $2k$   & $2k$ \\ \midrule
$\zeta_k^i, i=0,2$          & $2k+i$   & $D_{2k} \times \Z_2 \times \Z_2$ & $k$   & $i$   & $2-i$  \\
\midrule
$\zeta_k^1$          & $2k+1$   & $D_{2k} \times \Z_2$ & $k$   & $1$   & $1$  \\
\bottomrule
\end{tabular}
\end{table}

\section{Preliminaries on Eigenvalue Optimization}
\label{S:Prelim}

Let $M$ be a closed surface and $T\colon \Gamma\times M\to M$ be a faithful action of a finite group $\Gamma$ whose elements act by diffeomorphisms.  Denote by $\Met_T(M)$ the space of Riemannian metrics on $M$ invariant under this action.  By a standard averaging argument, $\Met_T(M)$ is nonempty, so the supremum
\begin{align*}
\Lambda_1^T(M):=\sup\{\bar{\lambda}_1(M,g)\mid g\in \Met_T(M)\}.
\end{align*}
is well-defined. 

We now restate \cite[Theorem 2.12]{KKMS}, which asserts that a metric realizing $\Lambda^T_1(M)$ is induced by a minimal immersion into $\Sph^n$.

\begin{theorem}
\label{thm:symm_crit_L}
Suppose that $g$ is a $T$-invariant metric, critical for $\bar\lambda_k$ in the class of $T$-invariant metrics on $M$. Then there exists an orthogonal representation $\rho\colon \Gamma\to O(n+1)$ and an equivariant branched minimal immersion $\Phi\colon (M,g)\to \Sph^{n}$ to the unit sphere such that the components of $\Phi$ are $\lambda_k(M,g)$-eigenfunctions and $\Phi^*g_{\Sph^n} = \alpha g$ for some $0<\alpha\in\R$.
\end{theorem}

In \cite[Theorem 7.7]{KKMS}, we gave criteria for the existence of a $\Lambda^T_1(M)$-maximizing metric.  We recall these criteria in Theorem \ref{thm:existence} below, after introducing some additional definitions from \cite{KKMS}.

\subsection{Topological degenerations}\label{tdeg}

Let $T:\Gamma\times M\to M$ be an action as above, and let $\fC = \cup c_i\subset M$ be a $\Gamma$-invariant union of disjoint, non-contractible, simple closed curves in $M$. The action of $\Gamma$ can be extended to the boundary of $M\setminus \fC$, and we denote by $M'_\fC$ the surface obtained by contracting each boundary component of  $M\setminus \fC$ to a point, and by $T'_\fC:\Gamma\times M'_\fC\to M'_\fC $ the corresponding action of $\Gamma$ on $M'_\fC$. We say that $(M'_\fC, T'_\fC)$ is obtained from $(M,T)$ by \emph{collapsing} $\fC$, and we denote by $P_\fC\subset M'_\fC$ the $T'_{\fC}$-invariant set of points corresponding to the contracted boundary components. 

Equivalently, $(M'_\fC,T'_{\fC})$ can be obtained by an equivariant surgery along a neighborhood of $\fC$. Indeed, let $\fU = \cup_i \fU_i$ be a $T$-invariant neighborhood of $\fC$, where the sets $\fU_i$ are disjoint tubular neighborhoods of the curves $c_i$. If the normal bundle to $c_i$ is orientable, then $\fU_i\approx\mathbb{S}^1\times \mathbb{B}^1$, so that $\del \fU_i$ is homeomorphic to two copies of $\mathbb{S}^1$. The surgery along $\fU_i$ then amounts to gluing two copies of a disk to $M\setminus \fU_i$ along the common boundary.

If the normal bundle to $c_i$ is not orientable, then $\fU_i$ is homeomorphic to the M\"obius strip, so that $\del \fU_i$ is a single circle. Surgery along $\fU_i$ in this case amounts to gluing a single copy of a disk to $M\setminus \fU_i$.

Since $\fU$ is $T$-invariant, one can arrange all the gluings to be equivariant, so that the restriction of $T$ to $M\setminus \fU$ and $\fU$ respectively extends to the glued disks as well. In this interpretation, the role of $P_\fC$ is played by the centers of disks glued to $M\setminus \fU$.

\begin{definition}
\label{def:degeneration}
Let $M$, $M'$ be closed, connected surfaces on which the group $\Gamma$ acts by $T,T'$ respectively. We say that $(M',T')$ is a \emph{topological degeneration} of $(M,T)$, and write
$(M',T')\prec (M,T)$, if there exists a non-empty $\Gamma$-invariant union of non-contractible disjoint simple closed curves $\fC\subset M$ such that $(M',T')$ is topologically equivalent to a $\Gamma$-invariant connected component $M'_{\fC,0}$ of $(M'_\fC,T'_\fC)$.
\end{definition}

\begin{definition}
\label{dedeg}
We say that $(M',T')$ is an {\em elementary degeneration} of $(M,T)$ if $(M',T')\prec (M,T)$ and there does not exist $(M'',T'')$ such that $(M',T')\prec (M'',T'')\prec (M,T)$.
\end{definition}

\begin{theorem}[Theorem 7.7, \cite{KKMS}]
\label{thm:existence}
For a surface with a finite group action $(M,T)$, set $\delta_{or} = 1$ if it has a $T$-invariant simple closed $2$-sided curve, $\delta_{or} = 0$ otherwise, and set $\delta_{nor} = 1$ if it has a $T$-invariant simple closed $1$-sided curve, $\delta_{nor} = 0$ otherwise. Then one has
\begin{equation}
\label{ineq:MM'}
\Lambda_1^T(M)\geq\max_{(M',T')\prec (M,T)} \left\{\Lambda_1^{T'}(M'), 8\pi\delta_{or}, 12 \pi \delta_{nor}\right\}.
\end{equation}

 Furthermore, if ~\eqref{ineq:MM'} is strict, then
 there exists a metric achieving $\Lambda_1^T(M)$, smooth up to possibly finitely many conical singularities.
\end{theorem}

\section{Main results with $D_2$-symmetry}
\label{SD2main}

\subsection{$D_2$-actions by basic reflections}
\label{ssD2chamber}
 Before proceeding further, we make precise the description in \ref{ssD2},  first recalling some standard notation \cite{Davis} (see also \cite[Section 5.7]{KKMS}) concerning groups generated by reflections on surfaces.

Let $M$ be a connected surface.  If $\Gamma$ is a finite group acting properly, smoothly, and effectively on $M$ and $\Gamma$ is generated by reflections, $\Gamma$ is called a \emph{reflection group}.  Given $x \in M$, let $R(x)$ denote the set of all reflections in $\Gamma$ fixing $x$.  A point $x \in M$ is called \emph{nonsingular} if $R(x) = \varnothing$.  A \emph{chamber} of $\Gamma$ on $M$ is the closure of a connected component of nonsingular points. 

Let $\Omega$ be a chamber, and denote by $V$ the set of reflections $v$ such that $\{v\} = R(x)$ for some $x \in \Omega$.  If $v \in V$, the set $\Omega^v : = M^v \cap \Omega$ is called a \emph{panel} of $\Omega$, and $V$ is the set of \emph{reflections through the panels of} $\Omega$.  The pair $(\Gamma, V)$ is called a \emph{reflection system}. 

Whenever $\Omega$ is a surface with corners, and $S$ the set of corners, an \emph{edge} of $\partial \Omega$ is the closure of a connected component of $\partial \Omega \setminus S$.

\begin{lemma}
\label{LD2quo}
Given data $(\Gamma,\Omega)$ consisting of
\begin{enumerate}[label=\emph{(\alph*)}]
\item The group $\Gamma = D_2$ with its set $V$ of non-identity elements; and 
\item A smooth disk with corners $\Omega$ equipped with a collection $(\Omega^v)_{v \in V}$ of \emph{panels}, such that each panel is a disjoint union of edges, and the interior of each edge belongs to exactly one panel,
\end{enumerate}
the quotient space $M$ of $\Gamma\times \Omega$ by the equivalence  relation $\sim$ defined by \\ $(g, x) \sim (h, y) \iff x = y$ and $g^{-1} h \in \langle v \in V : x \in \Omega^v\rangle$ satisfies: 
\begin{enumerate}[label=\emph{(\roman*)}]
\item $M$ is a connected, smooth surface, which admits an action $T$ of $\Gamma$ induced by the action on $\Gamma\times \Omega$ given by $g\cdot (h, p) = (gh , p)$.
\item The image in $M$ of each $\{\gamma \} \times \Omega$, $\gamma \in \Gamma$ is homeomorphic to $\Omega$.
\item $M$ is orientable if and only if $\Omega^v = \varnothing$ for some $v \in V$. 
\item $v \in V$ is a basic reflection if and only if $\Omega^v  \neq \varnothing$. 
\end{enumerate}
\end{lemma}
\begin{proof}
Notice that the construction of $M$ amounts to gluing four copies of $\Omega$, each indexed by an element of $\Gamma$, along $\Gamma \times \partial \Omega$; since the relation identifies the interior of each boundary edge with precisely one other boundary edge in another copy, items (i) and (ii) follow easily. 

Next, suppose $\Omega^\tau = \varnothing$ for some $\tau \in V$, and let $\rho_1, \rho_2, \tau$ denote the elements of $V$.  Then condition (b) implies $\partial \Omega$ is a $2k$-gon for some $k \geq 1$, with edges contained in $\Omega^{\rho_1}$ and $\Omega^{\rho_2}$ in alternating fashion.  Notice that the quotient $M$ can be realized by first identifying the boundary edges in $\Omega^{\rho_1}$ according to the rule above, and then making the $\rho_2$ identifications.  The first identification results in a pair of genus zero surfaces with $k$ boundary components, corresponding to the panel $\Omega^{\rho_2}$, and the second identifies these surfaces along the $\Omega^{\rho_2}$ boundary.
Consequently, $M$ is a closed, orientable surface of genus $k-1$, and $\rho_1, \rho_2$ are each basic reflections. 

Now, suppose $\Omega^\tau \neq \varnothing$ for some $\tau \in V$, and let $\rho_1, \rho_2$ be the other elements of $V$.   Modifying the steps in the previous paragraph shows that $M$ can be realized as the quotient of a connected orientable surface $M'$ with boundary components corresponding to the edges of $\Omega^\tau$; furthermore $\rho_1, \rho_2$ act on $M'$ as basic reflections, and $\tau = \rho_1\rho_2$ is orientation-preserving on $M'$.  

Let $e$ be an edge in $\Omega^\tau$.  If the neighboring edges of $e$ lie in the same panel, then $e$ corresponds in $M'$ to a $\tau$-invariant pair of loops $c, \tau(c)$.  Since $\tau$ is orientation-preserving, it follows that $M$ is nonorientable.  If the neighboring edges of $e$ lie in different panels, then $e$ corresponds in $M'$ to a single $\tau$-invariant loop $c$, on which $\rho_1, \rho_2$ act by reflections, whereas $\tau$ is an antipodal map.  Therefore, a neighborhood of the image of $e$ in $M$ is nonorientable.  Finally, it is easy to see that $M/\langle \tau \rangle$ is homeomorphic to the orientable, genus zero surface $M''$ obtained by gluing $\{1\}\times \Omega$ and $\{ \rho_1 \} \times \Omega$ by identifying points corresponding to $\Omega^{\rho_1}$, so $\tau$ is a basic reflection. 

The preceding paragraphs shows that (iii) and (iv) hold.  This completes the proof.
\end{proof}

\begin{remark}
If $M$ is a nonorientable surface constructed from $(\Gamma, \Omega)$ as in Lemma \ref{LD2quo}, it is not difficult to see that $M$ has Euler characteristic $4-k$, hence genus $k-2$, where $k$ is the number of edges of $\Omega$.
\end{remark}

\subsection{Existence of maximizing metrics}
Let $M$ be a nonorientable surface with an action $T$ of the group $\Gamma = D_2$ as in \ref{LD2quo}. 
In what follows, we classify the possible elementary degenerations $(M',T')\prec (M,T)$, 
then prove by induction that strict inequality holds in \eqref{ineq:MM'} for all such degenerations, establishing the existence of $\bar{\lambda}_1$-maximizing metrics. 

\begin{lemma}\label{deg.char}
    Let $(M,T)$ be a nonorientable surface with $D_2$-action as in Lemma \ref{LD2quo}, and let $(M',T')\prec (M,T)$ be an elementary degeneration. Then $(M',T')$ can be obtained from $(M,T)$ by collapsing a twisted oval for one of the reflections in $D_2$.
\end{lemma}
\begin{proof}
  By Definitions \ref{def:degeneration} and \ref{dedeg}, $(M',T')$ is a connected component of the space obtained by collapsing the $D_2$-orbit $\mathfrak{C}$ of some simple closed noncontractible curve $c\subset M$, and there is no intermediate $(M'',T'')$ for which $(M',T')\prec (M'',T'')\prec (M,T)$.

Let $\Omega \subset M$ be a chamber for the $D_2$-action $T$.  For such a degeneration, we can write $\mathfrak{C}$ as the $D_2$-orbit of a connected curve $\sigma$ in $\Omega$.  \emph{A priori}, $\sigma$ is an embedded loop in the interior of $\Omega$, or an arc with endpoints in $\partial \Omega$. 
However, $\Omega$ is a disk, so each loop in $\Omega$ is contractible, and the same holds for the $D_2$ orbit of such a loop.  Thus, $\sigma$ is an arc with endpoints in $\partial \Omega$.

If $\sigma$ is not contained in $\partial \Omega$, then $\Omega \setminus \sigma$ has connected components $\Omega_1$ and $\Omega_2$, and  collapsing $\mathfrak{C}$ leaves two components corresponding to the $D_2$-orbits of $\Omega_1$ and $\Omega_2$. Without loss of generality, let $\Omega_1$ be the orbit corresponding to $M'$. Next, note that $\Omega_2$ must contain at least one edge $\gamma$ of $\partial \Omega$, since $c$ 
is assumed to be noncontractible. Letting $(M'',T'')\prec (M,T)$ be the degeneration given by collapsing the $D_2$-orbit of $\gamma$, we deduce that either $(M'',T'')$ is equivalent to $(M',T')$, or there is a nontrivial chain $(M',T')\prec (M'',T'')\prec (M,T)$. Since $(M',T')$ is an \emph{elementary} degeneration, the first alternative must hold, so that $(M',T')\cong (M'',T'')$, and the conclusion of the lemma is satisfied.

It follows that, without loss of generality, we can assume that $\sigma\subset \partial \Omega$ is given by some union of edges, and for an elementary degeneration $(M',T')\prec (M,T)$, $\sigma$ must be a single edge, whose $D_2$-orbit $c$ is a fixed-point oval for one of the reflections $\tau$ of $D_2$. It remains to show that this oval must be twisted. 

As in the proof of Lemma \ref{LD2quo}, the orbit $c$ of $\sigma$ is untwisted if and only if the neighboring edges to $\sigma$ in $\partial \Omega$ lie in the same panel, say in $\Omega^{\rho_1}$.  It remains to show that, in this case, the degeneration $(M',T')\prec (M,T)$ given by collapsing $c$ is not elementary.

To this end, let $S$ be the connected component of $\Omega^{\tau} \cup \Omega^{\rho_1}$ containing $\sigma$, and $\alpha \subset S$ be an edge  meeting the boundary of $S$.  Then there are edges $\beta, \gamma$ neighboring $\alpha$ such that $\alpha \cup \beta  \subset \Omega^{\tau} \cup \Omega^{\rho_1}$, but  $\gamma \subset \Omega^{\rho_2}$. In this case, the orbit of $\alpha$ is a twisted oval, and collapsing this gives a degeneration $(M'',T'')\prec (M,T)$ whose marked fundamental chamber $\Omega''$ is obtained from $\Omega$ by deleting the edge $\alpha$. In $M''$, the $D_2$-orbit of the neighboring edge $\beta$ becomes a twisted oval, and collapsing this oval gives a further degeneration $(M''',T''')\prec (M'',T'')$ whose marked fundamental chamber $\Omega'''$ is obtained from $\Omega$ by removing any pair of neighboring edges in $S$. In particular, it follows that $(M''',T''')$ and $(M',T')$ have identical marked fundamental chambers, so we conclude that $(M''',T''')\cong (M',T')$, and therefore the degeneration $(M',T')\prec (M,T)$ is not elementary.
\end{proof}

Armed with this characterization of elementary degenerations, we can establish existence of maximizing metrics in this family by a simple induction on the nonorientable genus. The base case of the induction argument is straightforward.

\begin{claim}\label{rp2.clm}
Let $M = \mathbb{RP}^2$ and $T$ be a $D_2$-action on $M$ as in \ref{LD2quo}.  Then $\Lambda^T_1(M)=12\pi$ and is realized by the round metric.
\end{claim}
\begin{proof}
    In this case, $(M,T)$ must be obtained from the initial $D_2$-action on $\mathbb{S}^2$ by attaching a cross-cap at a fixed point with compatible $D_2$-action, and it is straightforward to check that the resulting $D_2$-action $T$ on $\mathbb{RP}^2$ can be realized by isometries for the round metric. Since the round metric $\mathbb{RP}^2$ is a global maximizer for $\bar{\lambda}_1$ on $\mathbb{RP}^2$ \cite{LiYau}, the claim follows immediately. 
\end{proof}

The induction step is based on a variant of the arguments of \cite[Section 3]{KSP}, adapted to the $D_2$-symmetric setting.

\begin{lemma}\label{twist.bd}
    If $(M',T')\prec (M,T)$ is an elementary degeneration as in Lemma \ref{deg.char}, and $\Lambda^{T'}_1(M')$ is realized by a smooth $T'$-invariant metric, then 
    \[\Lambda^T_1(M)>\Lambda^{T'}_1(M').\]
\end{lemma}
\begin{proof}
Suppose, to obtain a contradiction, that $\Lambda^{T'}_1(M')=\Lambda^{T}_1(M)$. 

By Lemma \ref{deg.char}, we see that $(M,T)$ can be obtained from $(M',T')$ by removing a $D_2$-invariant disk about a fixed point $p\in M'$ for the $D_2$-action on $M'$, and attaching a M\"obius band equipped with a compatible $D_2$-action. 

Moreover, it is clear that this $D_2$-action can be realized by isometries on any flat M\"obius band, so that the construction of \cite[Section 3]{KSP} can be applied to produce a family of $T$-invariant conformal classes $[g'_{p,\epsilon,L}]$ on $M$.

After replacing \cite[Proposition 2.2]{KSP} with its equivariant counterpart (cf. \cite[Lemma 8.10]{KKMS}), the proof of \cite[Proposition 3.2]{KSP} can then be applied to show that $\Lambda^{T'}_1(M')=\Lambda^T_1(M)$ implies the fixed-point $p$ must be a conical singularity for the metric $g_{max}$ on $M'$. But we have assumed that the metric $g_{max}$ is smooth, which is a contradiction.
\end{proof}

Combining these ingredients, we can now prove the following by an induction on the nonorientable genus.

\begin{theorem}
\label{TD2}
Let $M$ be a nonorientable surface and $T$ be an action of the group $D_2$ as in Lemma \ref{LD2quo}.  Then the following hold.
\begin{enumerate}[label=\emph{(\roman*)}]
\item There is a $T$-invariant metric $g$ on $M$ realizing $\Lambda^T_1(M)$. 
\item The metric $g$ is smooth and is induced by a $T$-invariant minimal embedding $\iota: M \rightarrow \Sph^4$ by $\lambda_1(M, g)$-eigenfunctions. 
\item The nonidentity elements $\rho_1,\rho_2, \rho_3 \in D_2$ are basic reflections on $M$.
\item The conclusions of Lemma \ref{Lgraphsph} hold, including $\area(M,g) < 8\pi$.
\item There is an orthogonal decomposition $\R^5 = E_1 \oplus E_2 \oplus E_3 \oplus F$ such that $E_i$ is $1$-dimensional for $i=1,2,3$, $F$ is $2$-dimensional, and
\[  \rho_i= I
\text{ on  } E_i \oplus F
\quad
\text{and}
\quad
\rho_i = - I
\text{ on } 
(E_i \oplus F)^\perp.
\]
\end{enumerate}
\end{theorem}

\begin{remark}
\label{RD2act}
By Theorem \ref{TD2}(v), one has
\begin{equation}
\label{ED2act}
\begin{aligned}
\rho_1(x_1, x_2,x_3, x_4, x_5) &= (x_1, -x_2, -x_3, x_4, x_5), 
\\
\rho_2(x_1, x_2,x_3, x_4, x_5) &= (-x_1, x_2, -x_3, x_4, x_5), 
\\
\rho_3(x_1, x_2,x_3, x_4, x_5) &= (-x_1, -x_2, x_3, x_4, x_5)
\end{aligned}
\end{equation}

in a suitable system of cartesian coordinates for $\R^5$.
\end{remark}

\begin{proof}
We make some preliminary observations which reduce the theorem to proving (i).  Theorem \ref{thm:symm_crit_L} and (i) imply the existence of a $T$-invariant full branched minimal immersion $\iota: M \rightarrow \Sph^n$ by $\lambda_1$-eigenfunctions, for some $n$, and Proposition \ref{mult.prop} and its proof show that $n =4$ and that (v) holds.  Next, the nonorientability of $M$ and Lemma \ref{LD2quo} imply (iii); combined with the preceding, the hypotheses of Lemma \ref{Lgraphsph} hold with $(M, \rho_i)$ for $i=1,2,3$ and $\iota$ as above, so the conclusions of \ref{Lgraphsph} hold as in (iv).  In particular, $\iota$ is a smooth embedding and free of branch points, implying (ii).  It remains to prove (i). 

We proceed by induction on the nonorientable genus. When $M$ has nonorientable genus one, this follows from Claim \ref{rp2.clm}. Suppose now that the result holds whenever $M$ has nonorientable genus $k$, and $\Lambda^T_1(M)\geq 12\pi$.

If $(M,T)$ is a pair as above with nonorientable genus $k+1$ and $(M',T')\prec (M,T)$, then either $M'$ has nonorientable genus $k$, or $M$ is orientable.  In the first case, it follows from the induction hypothesis that $\Lambda^{T'}_1(M')$ is realized by a smooth extremal metric. By Lemma \ref{twist.bd}, it then follows that
$$\Lambda^T_1(M)>\max_{(M', T') \prec (M, T)}\{\Lambda^{T'}_1(M'),12\pi\},$$
and therefore $\Lambda^T_1(M)$ is realized by a maximal metric $g_{max}$, by Theorem \ref{thm:existence}. 
In the second case, $\Lambda^{T'}_1(M')$ is also realized by a smooth maximal metric, by \cite[Theorem 8.7]{KKMS}, and again Lemma \ref{twist.bd} implies $\Lambda^T_1(M)$ is realized.
\end{proof}

\subsection{On the number of $D_2$-actions}
\label{ssCount}
We call a word $w$ in the letters $\rho_1, \rho_2, \rho_3$ \emph{admissible} if each letter appears at least once and no two adjacent letters are identical.  Recall from subsection \ref{ssD2chamber} that given an admissible word $w$ of length $k$ in the letters $\rho_1, \rho_2, \rho_3$, one can first construct a $k$-gon with sides labeled according to $w$, which in turn gives rise via Lemma \ref{LD2quo} to an action $T_w$ of $D_2$ on the surface $M = M_{k-3}$ with nonorientable genus $k-3$.  Whenever this is the case, let $\Gamma_w: = T_w(D_2)$ be the image of $D_2$ in $\Diff(M)$.

Here we define a set $\Wcal_k$ of admissible length $k$ words which grows exponentially in $k$ and has the property that for distinct $w, w' \in \Wcal_k$, the surface $M$ has no metric invariant with respect to both $T_w$ and $T_{w'}$.
Note that for $k>4$ the latter is equivalent to the fact that $\Gamma_w $ and $\Gamma_{w'} $ do not generate a finite subgroup of $\Diff(M)$.

Let $\Wcal_k$ be the maximal set of admissible words $w$ of length $k$ in the letters $\rho_1,\rho_2, \rho_3$ satisfying the following conditions:

\begin{enumerate}
    \item[(W1)] \label{L1}every even letter of $w$ is $\rho_1$;  
    \item[(W2)] \label{item2} $w$ starts with a segment $(\rho_2\rho_1)^{\alpha(k)}\rho_2$, where $\alpha(k)=\lceil k/6\rceil $; 
    \item[(W3)] $w$ contains the letter $\rho_2$ at least $\frac{1}{2}\left(\lceil k/2\rceil+\lceil k/6\rceil+1\right)>k/3$ times;
    \item[(W4)] the polygon corresponding to $w$ does not have order-$2$ symmetries;
    \item[(W5)] no two words $w,w'\in \Wcal_k$ differ by a cyclic permutation.
\end{enumerate}

We first estimate the cardinality $|\Wcal_k|$ of $\Wcal_k$.
\begin{lemma}
There exists $c>0$ such that $|\Wcal_k| \geq c2^{k/3}$  for each $k\geq 5$.
\end{lemma}
\begin{proof}
    We first count the number of words satisfying~(W1)-(W3). Since every other letter is $\rho_1$, the number of such words is the same as the number of words in letters $\rho_2,\rho_3$ of length $\lceil k/2\rceil$ starting with $\alpha(k)+1=\lceil k/6\rceil+1$ letters $\rho_2$ and at least one half of the remaining letters being $\rho_2$. Since the first $\lceil k/6\rceil+1$ letters are fixed, this in turn is the same as the number of words in $\rho_2,\rho_3$ of length $\lceil k/2\rceil - \lceil k/6\rceil-1$ in which at least half of the letters are $\rho_2$, which is easily seen to be $\geq \frac{1}{2}2^{\lceil k/2\rceil - \lceil k/6\rceil-1}$; hence, there are at least $c_0\cdot 2^{k/3}$ words satisfying ~(W1)-(W3), where we can take $c_0=1/8$.
    
   We now count how many of those words of length $\lceil k/2\rceil$ in $\rho_2$ and $\rho_3$ which could be cyclic permutations of one another. Note that if a word starting with $\alpha(k)+1$ letters $\rho_2$ is cyclic permutation of another such word, then either it starts with  $>\alpha(k)+1$ letters $\rho_2$, or it contains two disjoint segments of $\rho_2^{\alpha(k)+1}$. Note that the total number of words starting with $\rho_2^{\alpha(k)+1}\rho_3$ and having the majority of the remaining letters $\rho_2$ is again $\geq c_1 2^{k/3}$ for a suitable constant $c_1>0$. Among them, the number containing another segment  $\rho_2^{\alpha(k)+1}$ is $\leq C k2^{k/6}$, leaving $c_2 2^{k/3}$ of them that are not cyclic permutations of one another.

    Finally, let us count the number of words of this type, such that the corresponding $k$-gon has an order-$2$ symmetry. Any such order-$2$ symmetry is either a reflection or antipodal map. In either case, we again arrive at a situation that a symmetric word either starts with $>\alpha(k)+1$ letters $\rho_2$ or it contains two disjoint segments of $\rho_2^{\alpha(k)+1}$. Thus, the arguments of the previous paragraph still apply.
\end{proof}

We now show that if $w, w' \in \Wcal_k$ and $\Gamma_w, \Gamma_{w'}$ generate a finite subgroup of $\Diff(M)$, then the letters $w, w'$ are equal.  Our main tool is the following result, essentially contained in~\cite{Natanzon} and proved in Appendix \ref{SNat}.
\begin{theorem}
    \label{thm:Nat2}
    Let $\alpha$, $\beta$ be two non-commuting reflections on the orientable surface of genus $\gamma>1$ which generate a finite group. Then the total number of ovals of $\alpha$ and $\beta$ is at most $\gamma+4$.
\end{theorem}

Before proving the desired result, we discuss the geometric meaning of conditions~(W1)-(W3), which relates to the number of ovals of the lifts of the reflections $\rho_i\in \Gamma_w$ to the orientable double cover. Namely, let $\pi\colon \widetilde M\to M$ be the orientation cover, which can be identified with the orientation bundle of $M$. The elements $\rho_i\in\Gamma_w$ are lifted to $\widetilde M$ as $\widetilde\rho_i$ acting as $d\rho_i$ on the orientation bundle. Let $\tau$ be the involution on $\widetilde M$ switching the orientation of the fibers, so that $\pi$ is the factor by $\tau$. Let $O\subset M$ be an oval of $\rho_i$; then for $p\in O$ the differential $d_p\rho$ exchanges the orientation on $T_pM$, hence, on $\pi^{-1} (O)$ the involution $\widetilde\rho_i$ coincides with $\tau$. Finally, observe that $\pi^{-1}(O)$ has two connected components whenever $O$ is an untwisted oval, and one connected component whenever $O$ is a twisted oval.  As a result, if $(M, \rho_i)$ has species $[k-3,-\colon F,C_+,C_-]$, then $\tau\widetilde \rho_i$ has $2C_+ + C_-$ ovals. Thus, conditions~(W1)-(W3) imply that

\begin{itemize}
    \item $\rho_1$ has $\lfloor k/2\rfloor$ ovals;
    \item at least $k/6$ of those ovals are untwisted, so that $\tau\tilde\rho_1$ has at least $\lfloor k/2\rfloor +k/6 \geq \frac{4k-3}{6}>\frac{2(k-3)+4}{3}$ ovals;
    \item $\rho_2$ has at least $k/3$ untwisted ovals, so that $\tau\tilde\rho_2$ has at least \\$\frac{2k}{3}>\frac{2k-2}{3} = \frac{2(k-3)+4}{3}$ ovals.
\end{itemize}

\begin{proposition}
Suppose $k \geq 5$ and $w, w'$ are length-$k$ words with corresponding actions $T_w, T_{w'}$ on $M_{k-3}$ such that the group of diffeomorphisms generated by $\Gamma_w = T_w(D_2)$ and $\Gamma_{w'} = T_w(D_2)$ is finite.  Then
\begin{enumerate}[label=\emph{(\roman*)}]
\item If $w, w'$ satisfy (W1)-(W3),  any two elements of $\Gamma_w \cup \Gamma_{w'}$ commute.
\item If $w, w'$ satisfy (W1)-(W4), then $\rho_i = \rho'_i$ for $i=1,2,3$, so $\Gamma_w = \Gamma_{w'}$.
\item If $w, w'$ satisfy (W1)-(W5), then $w = w'$. 
\end{enumerate}
\end{proposition}
\begin{proof}
For (i), let $S=\{\rho_1, \rho_2, \rho'_1, \rho'_2\}$. Properties (W2) and (W3) imply the lift of each $\rho \in S$ has more than $(2\gamma+4)/3$ ovals, where $\gamma = k-3$ is the genus of $M$, hence also of its orientation cover $\widetilde{M}$.
Consequently, any two distinct elements of $S$ must commute,  for otherwise their lifts to $\widetilde{M}$ would together have more than $2 \frac{2\gamma+4}{3} > \gamma + 4$ ovals, contradicting Theorem \ref{thm:Nat2}.  Since $\rho_1, \rho_2$ generate $\Gamma_w$ and $\rho'_1, \rho_2'$ generate $\Gamma_{w'}$, item (i) follows. 

For (ii), let $\rho \in \{\rho_1, \rho_2\}$, and note by (i) that $\rho$ commutes with $\Gamma_{w'}$, so $\rho$ can also be viewed as an involution acting on the $k$-gon chamber $\Omega$ corresponding to $T'$; then (W4) implies this action must be trivial, which means that $\rho \in \Gamma_{w'}$, and hence $\Gamma_{w} \subset \Gamma_{w'}$.  Reversing the roles of $\Gamma_w$ and $\Gamma_{w'}$ in the above argument shows $\Gamma_{w'} \subset \Gamma_w$, so $\Gamma_w=\Gamma_{w'}$.  Finally, note that (W1)-(W3) imply that the basic reflections $\rho_i$ and $\rho_j$ have different species for $i \neq j$, so $\rho_i = \rho'_i$ for $i=1,2,3$, completing the proof of (ii).

Finally,  (iii) follows from (ii) since cyclic permutations of the same word yield the same action. 
\end{proof}

\section{The family $\Mtetk$ in $\Sph^4$}
\label{Stet}

\subsection{Elementary properties}
Let $k\geq 2$, and let $\Gamma_k = D_2 \rtimes D_{3k}$.  Recall that as a set, $\Gamma_k = D_2 \times D_{3k}$, and the group multiplication is given by
\begin{align}
\label{Esmult}
(n_1, h_1) \cdot (n_2, h_2) = (n_1 \phi_{h_1}(n_2), h_1h_2).
\end{align}

We now make explicit the action $T$ of $\Gamma_k$ on the nonorientable surface $M=M_k$ of genus $3k-2$ in the terminology of Lemma \ref{LD2quo}.  As in Lemma \ref{LD2quo}, $M$ is a quotient of $D_2 \times \Omega$, where $\Omega$ is a $3k$-gon, which may further be endowed with a standard $D_{3k}$-action as in \ref{ssD3}.  The group $\Gamma_k$ then acts on $D_2\times \Omega$ by
\begin{align}
\label{EGD3}
(\rho, \gamma) \cdot (\tau, p) = (\rho \phi_\gamma(\tau), \gamma \cdot p),
\end{align}
as is immediately verified using \eqref{Esmult} and the fact that $\phi$ is a homomorphism.  This action descends to the quotient $M$ of $D_2\times \Omega$ by the relation in Lemma \ref{LD2quo}, and we define this action to be $T$. 

\begin{lemma}
\label{LpermT}
The action $T$ of $\Gamma_k$ on $M$ acts on the set $P \subset M$ of four centers of the chambers for the $D_2$-action; moreover, each permutation of the elements of $P$ is realized by some element of $\Gamma_k$.
\end{lemma}
\begin{proof}
The first statement is clear.  For the second, let $P = \{p_1,p_2,p_3, p_4\}$, with $p_i$ identified with the center of the chamber corresponding to $\{\rho_i\} \times \Omega$ for $i=1,2,3$, and $p_4$ the center of the chmaber corresponding to $\{1\}\times\Omega$.  With this convention, the action $T$ leads to a homomorphism $\Phi : \Gamma_k \rightarrow S_4$, with $\Phi(g)$ the permutation induced by $g$ on the indices of $p_1,\ldots,p_4$.  From these definitions, it is clear that
\begin{equation}
\begin{gathered}
\label{Epermrho}
\Phi(\rho_1) = (14)(23),
\quad
\Phi(\rho_2) = (24)(13),
\quad
\Phi(\rho_3) = (34)(12),
\\
\Phi(\gamma_1) = (13),
\quad
\Phi(\gamma_2) = (23).
\end{gathered}
\end{equation}
The permutations in \eqref{Epermrho} generate $S_4$, completing the proof.
\end{proof}

\subsection{Maximizing metrics}

\begin{theorem}
\label{Ttetexist}
Let $k \geq 2$, $M$ be the surface with nonorientable genus $3k-2$, and $T$ be the action of $D_2 \rtimes D_{3k}$ on $M$ as in \ref{ssD3}.  Then the following hold.
\begin{enumerate}[label=\emph{(\roman*)}]
\item There is a $T$-invariant metric $g$ on $M$ realizing $\Lambda^T_1(M)$. 
\item The metric $g$ is smooth and is induced by a $T$-invariant minimal embedding $\iota: M \rightarrow \Sph^4$ by $\lambda_1(M, g)$-eigenfunctions. 
\item The nonidentity elements $\rho_1,\rho_2, \rho_3 \in D_2$ are basic reflections on $M$.
\item The conclusions of Lemma \ref{Lgraphsph} hold, including $\area(M,g) < 8\pi$.
\end{enumerate}
\end{theorem}
\begin{proof}
Exactly as in the proof of Theorem \ref{TD2}, it suffices to prove (i). 

Because $\Gamma_k$ contains the subgroup $D_2$ whose corresponding action is as in Lemma \ref{LD2quo}, the reasoning in the proof of Lemma \ref{deg.char} combined with the $\Gamma_k$-invariance shows that every elementary degeneration $(M',T')\prec (M,T)$ is realized by collapsing the $\Gamma_k$-orbit of one of the edges in the fundamental chamber depicted in Figure \ref{fig:chamber_Dk_nor}. 

In this case, however, we see that collapsing \emph{any} of the three edges leads to a degeneration $(M',T')\prec (M,T)$ for which $M'$ is disconnected, and therefore $\Lambda^{T'}_1(M')=0$. Moreover, these surfaces contain no $\Gamma_k$-invariant simple closed curves, so the strict form of the inequality \eqref{ineq:MM'} holds trivially, and (i) follows from Theorem \ref{thm:existence}.
\end{proof}

We now discuss the symmetries of the surfaces $\Mtetk \subset \Sph^4$ in more detail.  Before proceeding, we note that $\Mtetk$ is $D_2$-invariant in the sense of Section \ref{SD2main}, so as in Theorem \ref{TD2}(v), there is an orthogonal decomposition 
\[
\label{ER5split}
\R^5 = E_1 \oplus E_2 \oplus E_3 \oplus F
\]
such that $E_i$ is $1$-dimensional for $i=1,2,3$, $F$ is $2$-dimensional, and the nonidentity elements $\rho_1,\rho_2,\rho_3$ of $D_2$ satisfy
\[  \rho_i= I
\text{ on  } E_i \oplus F
\quad
\text{and}
\quad
\rho_i = - I
\text{ on } 
(E_i \oplus F)^\perp.
\]
We further choose cartesian coordinates for $\R^5$ so that \eqref{ED2act} holds.

\begin{lemma}
\label{Ltetsym}
Up to an ambient isometry of $\R^5$, the surface $\Mtetk$ and the action of the group $\Gamma_k = D_2 \rtimes D_{3k}$ satisfy the following. 
\begin{enumerate}[label=\emph{(\roman*)}]
\item The fixed-point set for the action of $D_2$ on $M_k$ is the set $F\cap M_k$, which consists of the $3k$-th roots of unity in the circle $F \cap \Sph^4$.
\item $D_{3k}$ acts on $F \subset \R^5$ via a standard dihedral action, satisfying
	\begin{align*}
	\gamma_1(z) = \bar{z},
	\quad
	\gamma_2(z) = e^{i \frac{2\pi}{3k}} \bar{z} . 
	\end{align*}
\item The set $P \subset M_k$ of centers of the chambers for the $D_2$-action is
	\[
	P = \textstyle{\frac{1}{\sqrt{3}}}\{(1,1,1), (1, -1, -1), (-1,1, -1), (-1, -1, 1)\} \subset E.
	\]

\item $D_{3k}$ acts on $E \subset \R^5$ via the action given on generators $\gamma_1, \gamma_2$ by
	\begin{align*}
    \gamma_1(x_1,x_2,x_3) = (x_3,x_2,x_1),
        \quad
	\gamma_2(x_1,x_2,x_3) = (x_1, x_3, x_2).
	\end{align*}	
\end{enumerate}
\end{lemma}
\begin{proof}
From the definitions in \ref{ssD3} and Lemma \ref{LD2quo}, the subset of $M_k$ fixed pointwise by the $D_2$-action is the set of $3k$ points corresponding to the vertices of any fundamental chamber $\Omega$ for the $D_2$-action; as in Theorem \ref{TD2}, these points lie in the great-circle $\Sph^1_F: = \Sph^4 \cap F$.  By the discussion in \ref{ssD3} and since the action $T$ is orthogonal, the generators $\gamma_1, \gamma_2$ of $D_{3k}$ act on $F$ by reflections, whose composition $\gamma_2 \gamma_1$ is a rotation by some angle $\theta$ such that $3k\theta$ is an integer multiple of $2\pi$.  To complete the proof of (i) and (ii), it suffices to prove $\theta =2\pi/(3k)$. 

If this were not the case, then there would exist vertices $p_1, p_2, p_k$ of the $3k$-gon $\Omega$ such that $p_1, p_2$ are adjacent in $\Omega$, while $p_k$ lies in the arc joining $p_1$ to $p_2$ in $\Sph^1_F$.  Then $p_1$ and $p_2$ lie on a fixed-point oval $O$ for one of the basic reflections $\rho_i$; furthermore, by Lemma \ref{Lgraphsph}(iv), the oval $O$ is a strictly convex curve in the $2$-sphere $\Sph^4 \cap (E_i \oplus F)$.  But then $p_k$ lies on the interior of the disk $D$ in $\Sph^4 \cap (E_i \oplus F)$ bounded by $O$, which is inconsistent with Lemma \ref{Lgraphsph}(v), which asserts the interior of $D$ is disjoint from image of $M$ under the nearest-point projection $\pi: M \rightarrow \Sph^4 \cap (E_i \oplus F)$.

For (iii), it is clear from the definitions in \ref{ssD3} that $\Gamma_k$ acts on $P$, and from the action of $\gamma_1, \gamma_2$ and their conjugates that $P$ lies in the $3$-dimensional subspace $F^\perp =E = E_1\oplus E_2 \oplus E_3$ of $\R^5$.  

Furthermore, by Lemma \ref{LpermT}, the distances between pairs of points in $P$ are independent of the choice of pair.  Since $P \subset E$, it follows by elementary geometry that $P$ consists of the vertices of a regular tetrahedron.  
Also, the action of the $\rho_i$ given by \eqref{Epermrho} implies 
\begin{align*}
\langle p_1, e_1\rangle = \langle p_4, e_1\rangle,
\quad
\langle p_2, e_1\rangle = \langle p_3, e_1\rangle,
\\
\langle p_2, e_2\rangle = \langle p_4, e_2\rangle,
\quad
\langle p_1, e_2\rangle = \langle p_3, e_2\rangle,
\\
\langle p_3, e_3\rangle = \langle p_4, e_3\rangle,
\quad
\langle p_1, e_3\rangle = \langle p_2, e_3\rangle,
\end{align*}
and combining this with the preceding and elementary geometry establishes that, without loss of generality, we may take 
\begin{equation*}
\begin{gathered}
p_4 = \frac{1}{\sqrt{3}}(1,1,1), 
\quad
p_1 = \frac{1}{\sqrt{3}}(1, -1, -1),
\\
p_2 = \frac{1}{\sqrt{3}}(-1,1,-1),
\quad
p_3 = \frac{1}{\sqrt{3}}(-1,-1,1),
\end{gathered}
\end{equation*}
proving (iii).

Finally, since $P$ lies fully in $E$, the action of $\gamma_1, \gamma_2$ on $E$ is determined by its action on $P$, which was determined above.  Thus (iv) follows from the preceding. 
\end{proof}

For later use in the study of blow-up limits, we record the following local $L^2$ estimate for the second fundamental forms of the minimal surfaces $M_k$.

\begin{lemma}
\label{Lsff}
There is a constant $C>0$ such that, for every $0<R<k$, 
\begin{align}
\label{Esff0}
\int_{B_{R/k}(q) \cap M_k} |\sff_{M_k} |^2 \leq CR 
\end{align}
for each $q \in M_k \cap F$, where $\sff_{M_k}$ is the second fundamental form of $M_k\subset \mathbb{R}^5$. 
\end{lemma}
\begin{proof}
Fix $R\in (0,\infty)$, and note that for any $k>R$, there is a constant $c>0$ independent of $k$ and $R$, and a collection of $\ell\geq \frac{ck}{R}$ points
$$q_1,\ldots,q_{\ell}\in F\cap M_k$$
such that the Euclidean balls $\{B_{R/k}(q_i)\}_{i=1}^{\ell}$ are mutually disjoint. By the $\Gamma_k$-invariance of $M_k$, it then follows that
\begin{eqnarray*}
    \int_{B_{R/k}(q)\cap M_k}|\sff_{M_k}|^2&=&\frac{1}{\ell}\sum_{j=1}^{\ell}\int_{B_{R/k}(q_j)\cap M_k}|\sff_{M_k}|^2\\
    &\leq &\frac{1}{\ell}\int_{M_k}|\sff_{M_k}|^2\\
    &\leq &\frac{R}{ck}\int_{M_k}|\sff_{M_k}|^2.
\end{eqnarray*}

On the other hand, it follows from the Gauss equation and the minimality of $M_k\subset \mathbb{S}^4$ that the second fundamental form $\sff_{M_k\subset\Sph^4}$ of $M_k$ in $\Sph^4$ satisfies $|\sff_{M_k\subset \Sph^4}|^2=2-2K_{M_k}$ where $K_{M_k}$ is the Gauss curvature of $M_k$, and therefore the full second fundamental form satisfies
$$|\sff_{M_k}|^2=2+|\sff_{M_k\subset \Sph^4}|^2=4-2K_{M_k}.$$
Since $M_k$ has Euler characteristic $4-3k$, this and the Gauss-Bonnet formula imply
\begin{align*}
\int_{M_k} |\sff_{M_k}|^2 = 4|M_k| -4\pi(4-3k)
\leq 28\pi k,
\end{align*}
where the last inequality follows by the area bound $|M_k| \leq 8\pi$.
Combining these estimates implies \eqref{Esff0}.
\end{proof}

\subsection{Varifold limits}\label{mtet.var}
We now describe the limiting behavior of the minimal surfaces $\Mtetk \subset \mathbb{S}^4$ as $k\to\infty$, first describing varifold limits in $\mathbb{S}^4$, then blowing up at suitable scales to obtain a new entire nonorientable minimal surface in $\mathbb{R}^4$.

We begin with the following general lemma concerning the structure of varifolds arising as large-topology limits of minimal surfaces in spheres with area below $8\pi$.

\begin{prop}
\label{Pvarifold}
If $M_k$ is a sequence of closed minimal surfaces immersed in the round sphere $\Sph^n$ with
\begin{align*}
\area(M_k) < 8\pi
\quad
\text{and}
\quad
\lim_{k \rightarrow \infty} \mathrm{dim}\, H_1(M_k) = \infty
\end{align*}
then some subsequence of the $M_k$ converges in the sense of varifolds to a limit varifold $V$ with mass $8\pi$, which must be one of the following:
\begin{enumerate}[label=\emph{(\roman*)}]
\item an equatorial two-sphere with multiplicity two;
\item a sum of four halves of equatorial two-spheres meeting along a common great-circle; or
\item a sum of two equatorial two-spheres meeting at a pair of antipodal points.
\end{enumerate}
\end{prop}
\begin{proof}
Denote by $V_k$ the integral varifold  associated to the surface $M_k$.  The compactness theorem for integral varifolds \cite{Allard} implies that, after passing to a subsequence, we have a  limit $V : = \lim V_k$ which is a stationary integral varifold, with mass  $\mass(V) \leq \limsup \mass(V_k)\leq 8\pi$.  

Similarly, denoting by $[V_k]$ the mod 2 flat chain associated to $V_k$ (see \cite{WhiteChain}), it follows from \cite[Theorem 1.1]{WhiteChain} that the chains $[V_k]$ converge to $[V]$ in the flat topology.  In particular, $\partial [V] = 0$ because $\partial [V_k] =0$ for each $k$.

If every point in $\spt V$ had density one, Allard's regularity theorem \cite{Allard} would imply that $M_k$ converges smoothly to $M$, contradicting the assumption that $\dim H_1(M_k) \rightarrow \infty$.  Thus, the singular set $\sing V$ is nonempty.

Next, we claim that each $p\in \sing V$ has density $\Theta_V(p)\geq 2$. Indeed, if $p\in \sing V$, note that any tangent cone $C_p$ to $V$ in $T_p\mathbb{S}^n$ can be realized as a limit of suitable rescalings of the varifolds $V_k$ about $p$, and another application of \cite[Theorem 1.1]{WhiteChain} implies that the associated mod 2 flat chain $[C_p]$ must satisfy $\partial [C_p]=0$. Since $C_p$ is a stationary 2-dimensional cone in $T_p\mathbb{S}^n\cong \mathbb{R}^n$, it must be the cone over a stationary geodesic network $W$ in $\mathbb{S}^{n-1}$, whose density $\Theta_W(x)$ at each point $x\in C_p\cap \mathbb{S}^{n-1}$ is given by $\frac{k}{2}$, where $k$ is the number of half-rays meeting at $x$, counted with multiplicity. Since $p\in \sing V$, $C_p$ is not a multiplicity-one plane, so there is at least one point $x$ with $\Theta_W(x)=\frac{k}{2}>1$. Moreover, since $C_p=C(W)$ defines a mod $2$ cycle, we see that $k\neq 3$. Thus, there is at least one point in the support of $W$ with $\Theta_W(x)\geq 2$, and a simple application of the monotonicity formula (see \cite[p. 591]{Simon}) then forces $\Theta_V(p)\geq\Theta_{C_p}(x)\geq 2$.

Now, consider the cone $C = C(V)$ over $V$ in $\R^{n+1}$.  For any $p \in \spt V$, it follows from \cite[Proposition 2.6(iv)]{Aiex}, the upper-semicontinuity of the density $\Theta_C$, and the mass bound $\mass(V) \leq 8\pi$ that
\begin{align*}
\Theta_V(p) = \Theta_C(p) \leq \Theta_C(0) = \mass(V)/4\pi \leq 2; 
\end{align*}
hence $\mass(V) = 8\pi$ and $\Theta_V(p) = 2$ for each $p \in \sing V$. Moreover, by a standard consequence of the monotonicity formula, we see that the \emph{spine} 
$$S(C)  = \{p \in \R^{n+1} : \Theta_C(p) = \Theta_C(0)=2\}$$
is a linear subspace of $\R^{n+1}$, and that $C$ is invariant under translations by elements of $S(C)$.  The above argument shows that $S(C)\cap \mathbb{S}^n= \sing V\neq \varnothing$.

We will  show that one of (i)-(iii) holds by showing that one of an equivalent list of properties holds for the cone $C$; these are that $C$ is 
\begin{enumerate}
\item a $3$-plane with multiplicity two; 
\item four half $3$-planes meeting along a $2$-dimensional subspace $P$
\item two $3$-planes meeting along a $1$-dimensional subspace.
\end{enumerate}
The preceding shows that $d:=  \dim S(C)\geq 1$.  This leaves three cases.

\emph{Case 1: $d = 3$}.  Then $S(C)$ is a $3$-plane, and $\Theta_C(0) = 2$ implies $C$ is supported on $S(C)$, so (1) holds.

\emph{Case 2: $d = 2$}.  Then $C$ splits as a product of $\R^2$ with a stationary one-dimensional cone $C'$ in $\R^{n-1}$ whose link with the unit sphere $\Sph^{n-2}$ 
must consist, since $\Theta_C(0) = 2$, of four distinct points $p_1, \dots, p_4$ with multiplicity one.  Then $C$ is a union of four half $3$-planes meeting along the $2$-dimensional subspace $S(C)$.

\emph{Case 3: $d = 1$}.  Then $C$ splits as a product of $\R$ with a stationary cone $C'$ in $\R^n$ whose link with the unit sphere $\Sph^{n-1}$ is a one-dimensional integral stationary varifold with empty singular set.  Since $\Theta_C(0) = 2$, it follows that this link is a sum of two disjoint great circles.  Then $C'$ is a union of two $2$-dimensional subspaces meeting at the origin, so $C$ is as in (3).
\end{proof}

In the case of the $\Gamma_k$-invariant minimal surfaces $M_k\subset \mathbb{S}^4$ arising from the $\bar{\lambda}_1$-extremal metrics on $\Mtetk$, we have the following. 

\begin{prop}
\label{Ptetlim}
Any subsequential varifold limit of the minimal surfaces $\Mtetk\subset \mathbb{S}^4$ consists of four half equatorial two-spheres meeting along a common great-circle, whose associated poles are the vertices of the regular tetrahedron.  

Equivalently, the cone $C(V)$ over any such limit varifold $V$ is of the form
\begin{align*}
C(V) = T \times \R^2,
\end{align*}
for $T\subset \R^3$ a cone meeting $\Sph^2$ in the set of vertices of a regular tetrahedron.
\end{prop}
\begin{proof}

The family $M_k$ in $\Sph^4$ satisfies the hypotheses of Proposition \ref{Pvarifold}, so any subsequential varifold limit $V$ of the varifolds $V_k$ associated to $M_k$ must satisfy one of the conclusions (i)-(iii) of \ref{Pvarifold}.

By Lemma \ref{Ltetsym}(i), the Hausdorff limit of the sets $M_k$ contains the great circle $F\cap \mathbb{S}^4$. Since varifold convergence of $V_k$ to $V$ implies that the supports $\spt V_k$ converge to $\spt V$ in the Hausdorff sense, it follows that the cone $C(V)$ over $V$ has the form 
$$C(V)=C'\times F \subset \R^5$$
where $C'\subset E= \R^3$ is a stationary $1$-dimensional cone. 

By Lemma \ref{Ltetsym}(iii), the surface $M_k$ contains the set
\[
P = \textstyle{\frac{1}{\sqrt{3}}}\{(1,1,1), (1, -1, -1), (-1,1, -1), (-1, -1, 1)\} \subset E
\]
for each $k$; by again using the Hausdorff convergence of the supports $\spt V_k $ to $\spt V$, it follows that $P\subset \spt V$.  Consequently, the cone $C(V)$ contains the cone $C(P) \times F$.  On the other hand, this cone has density $2$ at infinity; since $C(V)$ also has density two at infinity, it must be that $C(V) = C(P) \times F$.  This completes the proof. 
\end{proof}

\section{The surface $\MtetR$ in $\R^4$}
\label{SR4}
Here, we obtain a new highly-symmetric, singly-periodic nonorientable minimal surface embedded in $\R^4$, as a blow-up limit of the surfaces $\Mtetk\subset \Sph^4$.

Consider $\R^4$ with cartesian coordinates $x_1, x_2, x_3, x_4$, and write
\begin{align}
\label{ER4split}
\R^4 = E_1 \oplus E_2 \oplus E_3 \oplus S
\end{align}
with $E_i = \mathrm{span}(e_i)$ for $i=1,2,3$ and $S = \mathrm{span}(e_4)$.  Let $\Gamma$ be the subgroup of the affine isometries of $\mathbb{R}^4$ generated by the involutions
\begin{align*}
\gamma_1(x_1, x_2, x_3, x_4) &= (x_3, x_2, x_1, -x_4),
\\
\gamma_2 (x_1, x_2, x_3, x_4) &= (x_1, x_3, x_2, 1-x_4)
\\
\gamma_3(x_1, x_2, x_3, x_4) &= (x_1, -x_2, -x_3, 1-x_4).
\end{align*}
To describe further properties of $\Gamma$, we introduce some notation.  For $i=1,2,3$, let $P_i = E_i \oplus S$.  Also, define $T = \cup_{i=1}^4 R_i$ where each $R_i$ is a half-line in $E$, and the $R_i$ are defined by
\begin{align*}
\begin{gathered}
R_1 = \{ (t, -t, t) : t \geq 0\}
\quad
R_2 = \{(-t, t, -t): t \geq 0\}
\\
R_3 = \{(-t, -t, t) : t \geq 0\},
\quad
R_4 = \{(t,t,t): t \geq 0\}.
\end{gathered}
\end{align*}

The following properties of $\Gamma$ are easily verified:
\begin{itemize}
\item $\Gamma$ contains the translation $(\gamma_2\gamma_1)^3$ given by $x_4 \mapsto x_4 + 3$.
\item $\Gamma$ contains a subgroup isomorphic to $D_2$ generated by the reflections $\rho_i$, $i=1,2,3$ satisfying $\rho_i = I$ on $P_i$ and $\rho_i = - I$ on $P_i^\perp$.
\item $\Gamma$ acts on $T\times \mathbb{R}$.
\end{itemize}

\begin{theorem}
\label{TR42}
There is a complete, properly embedded nonorientable minimal surface $\MtetR$ in $\R^4$ with the following properties.
\begin{enumerate}[label=\emph{(\roman*)}]
\item $\MtetR$ is invariant under the group $\Gamma$.
\item $\MtetR$ has density two at infinity.
\item $\MtetR$ has a unique tangent cone at infinity, equal to $T \times \R$. 
\end{enumerate}
\end{theorem}

\begin{remark}
    Naturally, more could be said about the structure of the surface $\MtetR$; for instance, by adapting techniques from \cite{KKMS} and the present paper, it is straightforward to show that $\MtetR$ is a  symmetric minimal doubling of the two-plane $P_i$ for $i=1,2,3$, and each fixed-point oval $O \subset P_i$ is a convex curve.
\end{remark}

The proof of this theorem occupies the rest of the section; the surface $\MtetR$ arises as a blow-up limit of an appropriate family of rescalings of the $\Mtetk=M_k\subset \mathbb{S}^4$.

Recall by Lemma \ref{Ltetsym} that 
\begin{align*}
F\cap M_k = \{ (0, 0, 0, e^{i \frac{2\pi j}{3k}})\}_{j \in \Z},
\end{align*} 
and consider the dilated surfaces
\[
\widetilde{M}_k:=\frac{3k}{2\pi} (M_k-q)
\quad
\text{where}
\quad
q:=(0, 0, 0, 1, 0) \in F\cap M_k.
\]
Note that under the sequence of transformations $x\mapsto \frac{3k}{2\pi} (x-q)$, the spheres  $\frac{3k}{2\pi}(\mathbb{S}^4-q)$ converge smoothly on compact subsets to the tangent space
$$T_{q}\mathbb{S}^4=\mathrm{Span}\{e_1,e_2,e_3,e_5\},$$
as $k \rightarrow \infty$,
and the action of the elements $\rho_1,\rho_2,\rho_3,\gamma_1,\gamma_2$ of $\Gamma_k$ on $\mathbb{S}^4$ pass to the limit, to elements with the same names, satisfying
\begin{equation}
\begin{aligned}
\rho_1(x_1, x_2,x_3, x_5) &= (x_1, -x_2, -x_3, x_5), 
\\
\rho_2(x_1, x_2,x_3, x_5) &= (-x_1, x_2, -x_3, x_5), 
\\
\rho_3(x_1, x_2,x_3, x_5) &= (-x_1, -x_2, x_3, x_5),
\\
\gamma_1(x_1, x_2,x_3, x_5) &= (x_3,x_2,x_1,-x_5),
\\
\gamma_2(x_1, x_2,x_3, x_5) &= (x_1,x_3,x_2, 1 - x_5).
\end{aligned}
\end{equation}
We denote $\Gamma_\infty: = D_2 \rtimes D_\infty$ this limiting group, and note that up to relabeling, $\Gamma_\infty$ may be identified with $\Gamma$ as defined above.  For clarity, we now identify $T_q \Sph^4$ with $\R^4 = E_1 \oplus E_2 \oplus E_3 \oplus S$ as in \eqref{ER4split} by using the decomposition $\R^5 = E \oplus F$ from \eqref{ER5split}.

As a first step towards Theorem \ref{TR42}, we have the following.
\begin{prop}\label{ent.lim}
    After passing to a subsequence, the rescaled surfaces 
    \[
    \tilde{M}_k=\frac{3k}{2\pi} (M_k-q)
    \] converge smoothly on compact subsets of $\mathbb{R}^5$ to a complete, properly embedded minimal surface $\tilde{M}\subset T_{q}\mathbb{S}^4\cong \mathbb{R}^4$, invariant under the given action of $\Gamma_{\infty}$.
\end{prop}
\begin{proof}
An application of the monotonicity formula together with the area bound $|M_k|<8\pi$ yields a uniform density estimate
$$\limsup_{k\to\infty}\Theta(\tilde{M}_k,0,R)\leq 2;$$
that is,
\begin{equation}\label{ar.bd}
\limsup_{k\to\infty}|B_R(0)\cap \tilde{M}_k|\leq 2\pi R^2.
\end{equation}
By \eqref{ar.bd}, applying the compactness theorem for integral varifolds \cite{Allard} and passing to a subsequence if needed, we see that the surfaces $\tilde{M}_k$ converge in the varifold sense on compact subsets of $\mathbb{R}^5$ to a stationary varifold $V$ in $T_{q}\mathbb{S}^4\subset \R^5$ such that the following hold:
\begin{itemize}
    \item $V$ is invariant under the action of $\Gamma_{\infty}$.
    \item $V$ has density at infinity $\Theta_V(\infty):=\lim_{R\to\infty}\Theta_V(0,R)\leq 2.$
\end{itemize}
We will show now that the set
\[
\sing V=\{x\in \spt V\mid \Theta_V(x)>1\}
\]
where $V$ fails to be embedded with multiplicity one is empty; an application of Allard's regularity theorem \cite{Allard} then implies that $V$ is an embedded minimal surface with multiplicity one, and the convergence $\tilde{M}_k\to V$ is smooth on compact subsets.

Since $V$ arises as a limit of embedded minimal surfaces, it follows from the monotonicity formula and the mod-2 cycle property $\partial [V]=0$, just as in the proof of Proposition \ref{Pvarifold}, that 
\begin{itemize}
\item $\sing V=\{x\in \spt V \mid \Theta_V(x)=\Theta_V(\infty)=2\}$; 
\item $V$ is invariant under dilations about elements of $\sing V$; 
\item $\sing V$ is a linear subspace of $\R^4$; and
\item $V$ is invariant under translations by elements of $\sing V$.
\end{itemize}
Since $V$ is furthermore $\Gamma_\infty$-invariant, so is the subspace $\sing V$.  If, for the sake of a contradiction,  $\sing V$ is assumed nonempty, then it is easy to see that  $\sing V$ must be the $1$-dimensional subspace $S: = T_q \Sph^4 \cap F$.

Thus, $V$ has the form $V=C\times \mathbb{R}$, where $C$ is a nonflat stationary cone in $E =\mathbb{R}^3$, in which case we must have
\begin{align}
\label{Esffsing}
\lim_{k\to\infty}\sup_{y\in \tilde{M}_k\cap B_{\delta}(x)}|\sff_{\tilde{M}_k}|(y)=\infty
\end{align}
for every $x\in S$ and $\delta>0$ .

On the other hand, Lemma \ref{Lsff} and the fact that the integral of the second fundamental form is a scale-invariant quantity, it follows that
\begin{align}
\label{2ff.bd}
\int_{\tilde{M}_k\cap B_R(0)} |\sff_{\tilde{M}_k}|^2 \leq CR.
\end{align}
for each $R> 0$ and some $C$ independent of $R$ and $k$.

Next, note that \eqref{2ff.bd} and quantization results for the $L^2$-norm of the second fundamental form of minimal surfaces (for example \cite[Theorems 8.12 and 8.13]{White-notes}), imply that for each fixed $R>0$, there is a finite set $\mathcal{S} \subset  B_R(0)$ such that $|\sff_{\tilde{M}_k}|$ satisfies a uniform estimate of the form
\begin{equation}\label{pt.ii.bd}
|\sff_{\tilde{M}_k}|(x)\leq C(R,\delta)\text{ for }x\in (\tilde{M}_k\cap B_R(0))\setminus \bigcup_{p \in \mathcal{S}} B_\delta(p) 
\end{equation}
for each $k$ and each $\delta>0$.
But now \eqref{Esffsing} and \eqref{pt.ii.bd} are inconsistent as $k \rightarrow \infty$, since $S\cap B_R(0)$ is infinite, and this contradiction completes the proof.
\end{proof}

We next check that the limit $\tilde{M}$ is nonorientable.   Given that Lemma \ref{ent.lim} yields smooth convergence $\tilde{M}_k\to \tilde{M}$ on compact subsets of $\mathbb{R}^5$, it suffices to show that there is some $R>0$ independent of $k$ such that $\tilde{M}_k\cap B_R(0)$, or equivalently $M_k\cap B_{R/k}(q)$, is nonorientable for all $k\in \mathbb{N}$. 

To this end, let $p, q$ be two neighboring points in $F\cap M_k$, so that $p$ and $q$ form two vertices of a geodesic triangle $W_k$ formed by a pair of neighboring chambers $W_k(p), W_k(q)$ for the action of $\Gamma_k$ on $M_k$.  By definition of the $\Gamma_k$-action on $\Mtetk$, the edge $\sigma_k$ of $\partial W_k$ joining $p$ to $q$ lies in the fixed-set $M^{\rho_i}_k$ for some $i=1,2,3$, and its $D_2$-orbit $\beta_k$ is a twisted, geodesic fixed-point oval for $\rho_i$ on $M_k$.  Our goal now is to show that $\beta_k$ is homotopic to a curve contained in $B_{R/k}(q)$ for some fixed $R< \infty$ independent of $k$, so that $M_k \cap B_{R/k}(q)$ is nonorientable. 

\begin{claim}\label{loop.bd}
    There exists $R<\infty$ such that the twisted oval $\beta_k$ is homotopic to a closed curve $\tilde{\beta}_k$ contained in $B_{R/k}(q)$ for each $k \in \N$. 
\end{claim}
\begin{proof}
We first claim that the intrinsic distance $\mathrm{dist}_{M_k}(q, p)$ is bounded above by $C/k$ for some $C< \infty$.  Indeed, if this were not the case, then the corresponding limiting points  $0, \tilde{p} \in \tilde{M}$ would belong to distinct connected components of $\Mtilde$, forcing $\Mtilde$ to have infinitely many components by symmetry; since each connected component has area density $\geq 1$ at infinity, this would clearly violate the bound $\Theta(\Mtilde, \infty) \leq 2$.

It follows that there is a piecewise smooth path $\sigma$ joining $q$ to $p$ in $M_k$ with length $L(\sigma) \leq C/k$.
We next claim that $\sigma$ can be assumed to lie in $W_k$ without loss of generality.  To see this, first note that we can make an arbitrarily small perturbation such that $\sigma$ meets the boundary of any chamber for $\Gamma_k$ at most finitely times, and passes through no vertices except at the endpoints $q$ and $p$.  There is then a finite set of times $t_0 = 0< t_1< \cdots < t_n <t_{n+1} = L(\sigma)$ and reflections $\tau_1, \cdots \tau_n \in \Gamma_k$ such that $\sigma(t) \in \tau_j \cdots \tau_1 W_k$ for $t \in [t_j, t_{j+1}]$, while $\sigma(t_i) \in M^{\tau_i}_k$ and $\tau_n \cdots \tau_1 = 1$.  Defining $\tilde{\sigma}(t) : = \tau_1 \cdots \tau_j \sigma(t)$ on $[t_j, t_{j+1}]$ for $j=0, \dots, n$ then gives a piecewise smooth path joining $q$ to $p$ with image in $W_k$ such that $L(\tilde{\sigma})\leq L(\sigma)\leq C/k$, and we can simply set $\sigma=\tilde{\sigma}$.  

Since $\sigma$ is a path in the topological disk $W_k$ with the same endpoints as $\sigma_k$, it is clearly homotopic to $\sigma_k$ through maps fixing the endpoints. In particular, taking the $D_2$ orbit of $\sigma$ then gives a closed loop $\tilde{\beta}_k$ passing through $q$ homotopic to $\beta_k$, with length
$$L(\tilde{\beta}_k)=4L(\tilde{\sigma})\leq C'/k.$$
For any $y\in \tilde{\beta}_k$, it's clear that
$$|q-y|\leq dist_{M_k}(q,y)\leq L(\tilde{\beta}_k)\leq R/k$$
for some $R<\infty$, as claimed. Moreover, by mollification, we can take $\tilde{\beta}_k$ to be smooth.
\end{proof}

Putting these ingredients together, we can now complete the proof of Theorem \ref{TR42}.

\begin{proof}[Proof of Theorem \ref{TR42}]

By Proposition \ref{ent.lim}, we have already seen that, up to a possible subsequence, $\Mtilde_k$ converges smoothly on compact subsets of $\R^5$ to a complete, properly embedded minimal surface $\Mtilde$ invariant under the action of $\Gamma = \Gamma_\infty$.  We set $\MtetR = \Mtilde$. 

By Claim \ref{loop.bd}, there is some $R<\infty$ independent of $k$ such that $B_{R/k}(q)$ contains a closed loop homotopic to the twisted oval $\beta_k$, and therefore contains an embedded M\"obius strip. In particular, it follows that there is $R'<\infty$ such that $\tilde{M}_k\cap B_{R'}(0)$ is nonorientable for all $k\in\mathbb{N}$, and combining this with the smooth convergence $\tilde{M}_k\to \tilde{M}$ on compact subsets, it follows that $\tilde{M}$ is nonorientable as well.

To prove item (ii), note that we have already seen that
$$\lim_{R\to\infty}\Theta(\tilde{M},0,R)\leq 2,$$
and arguing as in the proof of Proposition \ref{Pvarifold}, we see that for any sequence $r_k\to 0$, any singular point $x\in spt(\tilde{C})$ in the varifold blow-down
$$\tilde{C}:=\lim_{r_k\to 0}r_k\tilde{M}$$
must have density $\Theta_{\tilde{C}}(x)\geq 2$, since $\tilde{C}$ is again a varifold limit of smooth minimal surfaces. Since $\tilde{C}$ is nontrivial, it follows in particular that 
$$\lim_{R\to\infty}\Theta(\tilde{M},0,R)=\Theta_{\tilde{C}}(0)=2,$$
proving (ii).

For (iii), note that the invariance of $\tilde{M}$ under the affine isometries $\gamma_1,\gamma_2,\gamma_3$ in $\mathbb{R}^4$ implies the invariance of $r_k\tilde{M}$ under their rescaled counterparts,
$$\gamma_1^{r_k}(x_1,x_2,x_3,x_4)=(x_3,x_2,x_1,-x_4),$$
$$\gamma_2^{r_k}(x_1,x_3,x_2,x_4)=(x_1,x_3,x_2,r_k-x_4),$$
$$\gamma_3^{r_k}(x_1,x_2,x_3,x_4)=(x_1,-x_2,-x_3,r_k-x_4).$$
Passing this invariance to the limit $r_k\to 0$, we deduce that any tangent cone at infinity $\tilde{C}=\lim_{r_k\to 0}r_k\tilde{M}$ must be invariant under translation along the $x_4$-axis, and have the form
$$\tilde{C}=W\times \mathbb{R},$$
where $W$ is a $1$-dimensional stationary integral cone in $\mathbb{R}^3$, with density two at infinity, invariant under the subgroup $G\leq O(3)$ generated by $(x_1,x_2,x_3)\mapsto (x_3,x_2,x_1)$, $(x_1,x_2,x_3)\mapsto (x_1,x_3,x_2)$, and $(x_1,x_2,x_3)\mapsto (x_1,-x_2,-x_3)$. In particular, the link of $W$ in $\Sph^2\subset \mathbb{R}^3$ is a set of four unit vectors invariant under this group, which sum to zero.  

It remains to show that this invariance forces 
$$W=P=\left\{\frac{1}{\sqrt{3}}(1,1,1),\frac{1}{\sqrt{3}}(1,-1,-1),\frac{1}{\sqrt{3}}(-1,1,-1),\frac{1}{\sqrt{3}}(-1,-1,1)\right\},$$
which is straightforward. Namely, it is easy to see that the reflections $(x_1,x_2,x_3)\mapsto (x_3,x_2,x_1)$ and $(x_1,x_2,x_3)\mapsto (x_1,x_3,x_2)$ generate the action of $S_3$ on $\mathbb{R}^3$ by permutation of the standard basis vectors, and any four-element subset $W\subset\mathbb{S}^2$ invariant under this action must include the fixed point $\frac{1}{\sqrt{3}}(1,1,1)$, whose $G$-orbit is indeed $P$. 

Thus, any tangent cone $\tilde{C}$ to $\tilde{M}$ at infinity must indeed have the form $\tilde{C}=T\times \mathbb{R}$, where $T$ is the cone over $P$, completing the proof of (iii).
\end{proof}

\section{Other Families}
\label{SOther}
\subsection{Existence for $\eta^v_k$, $\eta^c_k$, and $\eta^e_k$}

Next, for each $k\in \mathbb{N}$, we consider $(M_k,T)$ corresponding to one of the families $\eta^v_k$, $\eta^c_k$, or $\eta^e_k$ of nonorientable surfaces equipped with the action of the group $\Gamma_k=(D_k\times D_2)\rtimes \mathbb{Z}_2$ described in Section \ref{ssDkD2}. Recall that fundamental chambers $V_k$ for $\eta_k^v$ and $\eta_k^c$ can be found in Figure \ref{fig:chamber_Dk_nor}, where they can be identified with the quadrilaterals with edges labeled by $\gamma$, $\gamma_1$, $\rho_3$, and $\rho_2$, and the same is true for $\eta_k^e$.

For all three of these families, the elementary degenerations are relatively easy to describe. If $(M,T)$ is one of $\eta^{v}_k$, $\eta^{c}_k$, or $\eta_k^e$, and $\mathfrak{C}$ is a minimal $\Gamma_k$-invariant union of noncontractible simple closed curves in $M$, then $\mathfrak{C}$ can be written as the $\Gamma_k$ orbit of some arc $\sigma$ in the quadrilateral fundamental chamber $V$. Arguing as in the proof of Lemma \ref{deg.char}, it is easy to see that this $\sigma$ can be taken without loss of generality to be an edge in $\partial V$, fixed by one of $\rho_3$, $\rho_2$, $\gamma_1$, or $\gamma$. 

In the case $(M,T)=\eta^{v}_k$, collapsing the $\Gamma_k$-orbit of the $\gamma_1$-labeled edge results in a degeneration $(M',T')\prec (M,T)$ where $M'$ is disconnected, so that $\Lambda^{T'}_1(M')=0$, and the same is true for the $\gamma$-labeled and $\rho_2$-labeled edges. Thus, the only elementary degeneration $(M',T')\prec (M,T)$ for which $\Lambda^{T'}_1(M')>0$ is given by collapsing the collection of twisted fixed point ovals for $\tau$ corresponding to the $\tau$-labeled edge, for which $M'\approx N_{k-1}$. By another application of Theorem \ref{thm:existence}, we can then prove the following.

\begin{theorem}
\label{Tetatimes}
Let $M = \eta^{v}_k$ and $T$ be the action of $(D_k \times D_2) \rtimes \Z_2$ on $M$.  Then the following hold.
\begin{enumerate}[label=\emph{(\roman*)}]
\item There is a $T$-invariant metric $g$ on $M$ realizing $\Lambda^T_1(M)$. 
\item The metric $g$ is smooth and is induced by a $T$-invariant minimal embedding $\iota: M \rightarrow \Sph^4$ by $\lambda_1(M, g)$-eigenfunctions.
\item Conclusions (iii)-(v) of Theorem \ref{TD2} hold.
\item $D_k$ acts trivially on $E$, and as a standard $D_k$ action on $F$, given by
	\begin{align*}
	\gamma_1(z) = \bar{z},
	\quad
	\gamma_2(z) = e^{i \frac{2\pi}{k}} \bar{z},
	\end{align*}
	up to an ambient isometry.  Here we use the notation from \ref{TD2}.
\end{enumerate}
\end{theorem}
\begin{proof}
As in the proof of Theorem \ref{TD2}, we may reduce proving items (i)-(iii) to proving (i). 

In this case, we see that $(M,T)$ contains no $\Gamma_k$-invariant simple closed curves, so by Theorem \ref{thm:existence}, it suffices to show that $\Lambda^T_1(M)>\Lambda^{T'}_1(M')$ for any elementary degeneration $(M',T')\prec (M,T)$. 
    
    By the discussion above, the only elementary degeneration for which $\Lambda^{T'}_1(M')>0$ is given by collapsing a collection of symmetric cross-caps, with $M'\approx N_{k-1}$. In this case, it follows from the analysis of \cite[Section 8]{KKMS} that $\Lambda^{T'}_1(M')$ is realized by a $\bar{\lambda}_1$-maximizing metric, induced by a minimal embedding of $N_{k-1}$ into $\mathbb{S}^3$.

    Since the degeneration $(M',T')\prec (M,T)$ is achieved by collapsing cross-caps, and $\Lambda^{T'}_1(M')$ is realized by a smooth maximal metric, we can argue exactly as in the proof of Lemma \ref{twist.bd} to deduce that 
    $$\Lambda^T_1(M)>\Lambda^{T'}_1(M'),$$
    completing the proof of (i).
    
    Finally, we prove (iv).  Fix $i \in \{1,2\}$ and consider the fixed-point set $M^{\gamma_i}$.  From the discussion in \ref{ssD2}, $M$ is realized by gluing four copies of a fundamental chamber by identifying the labels on the edges, and $M^{\gamma_i}$ bisects any such chamber, separating it into two pieces.  Furthermore, it is easy to see from the identifications of the chambers via the $\rho_1, \rho_2, \rho_3$ labels that $M^{\gamma_i}$ separates $M$ into two pieces.  

Since $\gamma_i$ is an involution, the first eigenspace $\Ecal$ on $M$ admits a direct sum decomposition $\Ecal = \Ecal^+ \oplus \Ecal^-$ into $\gamma_i$-even and $\gamma_i$-odd parts; since any $\gamma_i$-odd function vanishes on $M^{\gamma_i}$, the Courant nodal domain theorem and the fact that $M^{\gamma_i}$ separates $M$ implies $\dim \Ecal_- \leq 1$.  Because $M$ admits a full immersion into $\R^5$ by first eigenfunctions, it follows that $\dim \Ecal_- = 1$ and $\dim \Ecal_+ = 4$, so $\gamma_i$ acts by reflection through a hyperplane of $\R^5$.

By construction $D_k$ commutes with $D_2$, hence $\gamma_i$ fixes $E$, and acts on $F$.  Under the projection $\pi_3: M \rightarrow (E_3 \oplus F) \cap \Sph^4$, the chamber for $M$ corresponding to $\rho_3$ can be identified with a subset of the upper hemisphere of $\Sph^2$.  From the preceding, we can now assume, up to rotation, that the action of $\gamma_1$ on $\Sph^2$ is by $z \mapsto \bar{z}$.  Since $M$ is embedded and the fixed-point ovals for $\rho_3$ are convex curves by \ref{Lgraphsph}(iv), it follows easily that the fixed-point set for $\gamma_2$ on $(E_3\oplus F) \cap \Sph^4$ is a great circle meeting $\{z= 0\} \cap \Sph^2$ at angle $\pi/k$.  Without loss of generality, we may assume $\gamma_2$ acts on $E_3 \oplus F$ by $(x_3, z) \mapsto (x_3, e^{i \frac{2\pi}{k}})$, and (iv) follows. 
\end{proof}

We now turn to the case $M = \eta^{c}_k$ and its action $T$ of $(D_k \times D_2)\rtimes \Z_2$.  Regarding elementary degenerations, collapsing the $\Gamma_k$-orbit of the edge labeled by $\gamma$, $\rho_2$, or $\gamma_1$ results in a degeneration $(M',T')\prec (M,T)$ with $M'$ disconnected, so that $\Lambda^{T'}_1(M')=0$. In this case, the only elementary degeneration for which $\Lambda^{T'}_1(M')>0$ is given by collapsing the collection of untwisted fixed ovals for $\rho_3$ corresponding to the $\rho_3$-labeled edge, again with $M'\approx N_{k-1}$.

\begin{proposition}\label{etabull.ex}
    If $(M,T)=\eta^{c}_k$ with the given action of $\Gamma_k=(D_k\times D_2)\rtimes \mathbb{Z}_2$, the supremum $\Lambda^T_1(M)$ is realized by a $\Gamma_k$-invariant extremal metric $g_{max}$.
\end{proposition}
\begin{proof}
    Again, we see that $(M,T)$ contains no $\Gamma_k$-invariant simple closed curves, so it suffices to prove that $\Lambda^T_1(M)>\Lambda^{T'}_1(M')$ for the elementary degenerations described above.

    In particular, it is enough to consider the degeneration $(M',T')\prec (M,T)$ given by collapsing the pair of untwisted ovals for $\rho_3$, so that $M'=N_{k-1}$. In the language of \cite[Theorem 8.1]{KKMS}, this degeneration leaves a collapsed set $P=\{p,\rho_3 p\}\cup\{\rho_1p,\rho_2p\}$ in $N_{k-1}$, where $p$ is a fixed point for the subgroup $D_k\leq \Gamma_k$. If the desired strict inequality failed, it then follows from \cite[Theorem 8.1]{KKMS} that there is some map $\Psi: (N_{k-1},g_{max})\to \mathbb{S}^n$ by first eigenfunctions for the metric $g_{max}$ realizing $\Lambda^{T'}_1(M')$ such that $\Psi(p)=\Psi(\tau p)$. 

    In that case, it follows from an application of \cite[Proposition 8.5]{KKMS} that we must have $k=2$, with $(N_1,g_{max})$ homothetic to the square torus $[-1,1]^2/2\mathbb{Z}^2$, where $\Gamma_1$ acts such that $\rho_3(x,y)=-(x,y)$ and, without loss of generality, $p=(1/2,1/2)$. In this case, we see that any sphere-valued map $\Psi$ by first eigenfunctions must have the form $\Psi(x,y)=(a e^{2\pi i x}, b e^{2\pi i y})$ for some $(a,b)\in \mathbb{S}^1$, and for each of these, $\Psi(\tau p)\neq \Psi(p)$. Thus, we must have $\Lambda^T_1(M)>\Lambda^{T'}_1(M')$ in this case as well, completing the proof. 
\end{proof}

For $M=\eta_k^e$ with the given action of $\Gamma_k=(D_k \times D_2)\rtimes \Z_2$, collapsing the $\Gamma_k$-orbit of the edge labeled by $\gamma$, $\gamma_1$, or $\rho_2$ again results in a disconnected degeneration $(M',T')\prec (M,T)$, with $\Lambda_1^{T'}(M')=0$. In this case, the elementary degeneration for which $\Lambda_1^{T'}(M')>0$ is given by collapsing either the collection of untwisted ovals for $\rho_3$ corresponding to the $\rho_3$-labeled edge, so that $M'\approx N_{k-1}$.

\begin{proposition}\label{etae.ex}
    If $(M,T)=\eta^{e}_k$ with the given action of $\Gamma_k=(D_k\times D_2)\rtimes \mathbb{Z}_2$, the supremum $\Lambda^T_1(M)$ is realized by a $\Gamma_k$-invariant extremal metric $g_{max}$.
\end{proposition}
\begin{proof}
Again, it suffices to prove that $\Lambda^T_1(M)>\Lambda^{T'}_1(M')$ for the elementary degeneration to $N_{k-1}$ given by collapsing the $\rho_3$-labeled edge. In this case, the collapsed set in $N_{k-1}$ is given by the $D_k$-orbit of $P=\{p,\rho_3 p\}$ for some point $p$ in $N_{k-1}$ fixed by $\rho_2$.

As in the proof of Proposition \ref{etabull.ex}, an application of \cite[Proposition 8.5]{KKMS} then shows that we must have $k=2$, with $(N_1,g_{max})$ homothetic to the square torus $[-1,1]^2/2\mathbb{Z}^2$, equipped with some map $\Psi: N_1\to \mathbb{S}^n$ by first eigenfunctions such that $\Psi(\rho_3p)=\Psi(p)$ and $\Psi(\rho_3 q)=\Psi(q)$ for a suitable pair of points $p,q$ fixed by $\rho_2$ and $\rho_1$, respectively. Once again, it is easy to see directly that no such map exists, so we must have $\Lambda_1^T(M)>\Lambda_1^{T'}(M')$, completing the proof.
\end{proof}

\section{Main Results with $\Z_2$-symmetry}
\label{SZ2main}

In this section we study non-orientable BRS with a single basic reflection. In this case, the analysis of topological degenerations is more cumbersome due to the fact the basic reflection does not separate the surface, so there is no convenient fundamental domain for the action.

\subsection{Neighborhoods of equivariant circles in reflection surfaces}

Following \cite{Dugger}, we introduce notation for describing circles equipped with involutions, and their tubular neighborhoods in  reflection surfaces.  Let
\begin{align*}
S^{1,0},
\quad
S^{1,1}, 
\quad
\text{and}
\quad
S^1_a
\end{align*}
denote circles equipped with following involutions, respectively: the trivial involution, a reflection across a diameter, and the antipodal map.  

If such a circle $S$ is embedded in a reflection surface $(M, \tau)$, then as in \cite[p. 936-937]{Dugger}, a small $\tau$-invariant neighborhood $U$ of $S$ has one of seven possible equivariant homeomorphism types, enumerated as follows.
\begin{itemize}
\item  $S^{1,0}$-antitube:   $S$  is an untwisted $S^{1, 0}$ and $\tau$ exchanges the boundary circles of the cylinder $U$.
\item $S^{1,1}$-antitube: $S$ is an untwisted $S^{1,1}$, and $\tau$ exchanges the boundary components of the cylinder $U$.
\item $S^{1,1}$-tube: $S$ is an untwisted $S^{1,1}$, and $\tau$ preserves the boundary components of the cylinder $U$.
\item $S^1(M)$: $S$ is a twisted $S^{1,0}$ and $U$ is a M{\"o}bius band with an $S^{1}_a$ as boundary. 
\item $S^{1,1}(M)$: $S$ is a twisted $S^{1,1}$ and $U$ is a M{\"o}bius band with an $S^{1,1}$ as boundary.  The fixed-point set $U^\tau$ consists of a segment joining the fixed-points on the boundary $S^{1,1}$, and an isolated fixed-point. 
\item $S^1_a$-antitube: $S$ is an untwisted $S^1_a$, and $\tau$ exchanges the boundary components of the cylinder $U$.
\item $S^1_a$-tube: $S$ is an untwisted $S^1_a$, and $\tau$ preserves the boundary components of the cylinder $U$.
\end{itemize}

The prefix ``anti" indicates that $\tau$ exchanges the components of $\partial U$.  See Figures \ref{Ftubes} and \ref{Fmob} for schematics for some of these actions. 
 
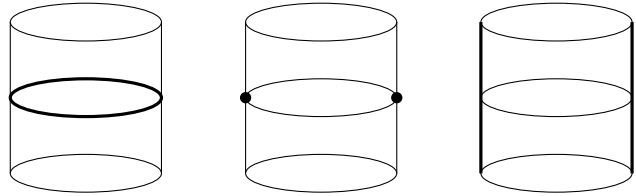
\begin{figure}[ht] 
\begin{tikzpicture}[scale =.25]
\draw (0, 0) ellipse (4 and 1);
\draw (0, 8) ellipse (4 and 1);
\draw[very thick] (0, 4) ellipse (4 and 1);
\draw (-4, 0)--(-4, 8);
\draw (4, 0)--(4, 8);
\end{tikzpicture}
\hspace{.3in}
\begin{tikzpicture}[scale =.25]
\draw (0, 0) ellipse (4 and 1);
\draw (0, 8) ellipse (4 and 1);
\draw (0, 4) ellipse (4 and 1);
\draw (-4, 0)--(-4, 8);
\draw (4, 0)--(4, 8);
\node at (-4, 4)[circle,fill,inner sep=1.5pt]{};
\node at (4, 4)[circle,fill,inner sep=1.5pt]{};
\end{tikzpicture}
\hspace{.3in}
\begin{tikzpicture}[scale =.25]
\draw (0, 0) ellipse (4 and 1);
\draw (0, 8) ellipse (4 and 1);
\draw (0, 4) ellipse (4 and 1);
\draw (-4, 0)--(-4, 8);
\draw (4, 0)--(4, 8);
\draw[very thick](-4, 0)--(-4, 8);
\draw[very thick](4, 0)--(4, 8);
\end{tikzpicture}
\caption{Depictions of an $S^{1,0}$-antitube (left), an $S^{1, 1}$-antitube (middle), and an $S^{1,1}$-tube (right).  Fixed-point sets of the corresponding involutions are depicted in bold.}
\label{Ftubes}
\end{figure}

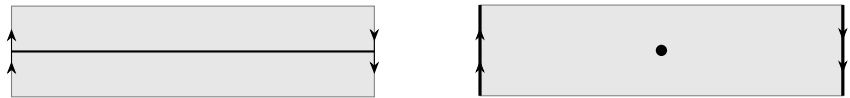
\begin{figure}[ht]
\begin{tikzpicture}[scale =.6]
	\filldraw[fill=lightgray!75,semitransparent] (-4, -1) rectangle (4, 1);
	\draw[>=Stealth, >->] (-4, -.5)--(-4, .5); 
	\draw[>=Stealth, >->] (4, .5)--(4, -.5); 
	\draw[thick] (-4, 0)--(4, 0);
\end{tikzpicture}
\hspace{.4in}
\begin{tikzpicture}[scale =.6]
	\filldraw[fill=lightgray!75,semitransparent] (-4, -1) rectangle (4, 1);
	\draw[>=Stealth, >->] (-4, -.5)--(-4, .5); 
	\draw[>=Stealth, >->] (4, .5)--(4, -.5); 
	\draw[very thick] (-4, -1)--(-4, 1); 
	\draw[very thick] (4, -1)--(4, 1); 
	  \node at (0, 0)[circle,fill,inner sep=1.5pt]{};
\end{tikzpicture}
\caption{The M{\"o}bius bands $S^1(M)$ (left) and $S^{1,1}(M)$ (right).  In each case, $M$ is obtained by identifying the sides of the rectangle as indicated.  For $S^1(M)$, the involution is induced by the reflection fixing the bold segment, and for $S^{1,1}(M)$, the involution is induced by the reflection through the center point, which also fixes the bold segment where the identification occurs.}
\label{Fmob}
\end{figure}

\subsection{Equivariant surgeries}
\label{sssurg}
There are several kinds of equivariant surgery operations we will need to apply on surfaces equipped with involutions.  In practice, we will apply these surgeries on basic reflection surfaces, or the surfaces  $T^{\mathrm{sk}}_g$, where for $g\geq 0$, the space $T^{\mathrm{sk}}_g$ consists of a genus-$g$ orientable surface $M$ with an involution $\tau$ fixing $2g+2$ isolated points.  For visualization purposes, it is convenient to imagine $M$ as being embedded in $\R^3$ and $\tau$ as being induced by rotation by 180 degrees about a line (the ``skewer'') meeting $M$ in $2g+2$ points.  See Figure \ref{Fhandlebody}.

We first enumerate the surgeries which involve attaching a M{\"o}bius band:
\begin{itemize}
\item $[FMO_-]$: remove a small neighborhood of an isolated fixed-point, leaving an $S^1_a$ on the boundary, and attach an $S^{1}(M)$. 
\item $[O_+ MO_-F]$: remove an $S^{1,1}$-cap along a segment of an untwisted oval, leaving a $S^{1,1}$ along the boundary, and attach an $S^{1,1}(M)$.
\item $[O_- M O_+F]$: remove a small disk along a portion of a twisted oval, leaving a $S^{1,1}$ along the boundary, and attach an $S^{1,1}(M)$.  
\end{itemize}

If $(M', \tau)$ is a surface equipped with an involution and $[*]$ is one of these three surgery operations which is defined on $M$, we write
\begin{align*}
M = M' + [*] 
\end{align*}
to denote the reflection surface $(M,\tau)$ resulting from the surgery.  

The reasons for the above names are as follows: the surgery $[FMO_-]$ replaces an isolated fixed-point with a twisted oval $O_-$; while $[O_+ MO_-F]$ replaces an untwisted oval $O_+$ with a twisted oval $O_-$ and an isolated fixed-point, and $[O_- MO_+F]$ replaces a twisted oval $O_-$ with an untwisted oval $O_+$ and an isolated fixed-point. In the latter two cases, that the resulting oval is actually twisted (respectively untwisted) follows from a simple and direct verification. 
 In each case, the $M$ signifies the attaching of a M{\"o}bius band. 

\begin{lemma}
\label{Lsurgbrs}
Let $[X]$ be one of the surgery operations
\begin{align*}
    [FMO_-], 
    \quad
    [O_+MO_-F],
    \quad
    [O_-MO_+F].
\end{align*}
If $M$ is a basic reflection surface, then so is $M+[X]$.
\end{lemma}
\begin{proof}
In each case, the fixed-point set of $N=M+[X]$ has at least one fixed-point oval, hence $N$ is a reflection surface.  It remains to prove the quotient $\pi(N)$ of $N$ by the involution has genus zero.  

Since $M$ is a basic reflection surface, the quotient $\pi(M)$ has genus zero, and $N$ is obtained from $M$ by removing an equivariant disk and replacing it with an equivariant M{\"o}bius band.  In each case, the quotient of the M{\"o}bius band by the involution is a disk.  It follows that $\pi(N)$ is obtained by attaching a disk to a genus-zero surface, so $\pi(N)$ has genus zero as well. 
\end{proof}

\subsection{Main results with $\Z_2$-symmetry}
\label{ssZ2main}
In this subsection, let $(M, \tau)$ be a basic reflection surface and $T$ be the corresponding action of $\Z_2$ as in \ref{ssZ2}.  We consider the spherical, elliptic, and hyperelliptic cases in turn.
\begin{prop}
\label{Pelem0}
Let $(M, T)$ be a nonorientable spherical basic reflection surface and $(M', T')\prec (M, T)$ be an elementary degeneration.  Then $(M',T')$ is a basic reflection surface (possibly, orientable), or is the space $T^{\mathrm{sk}}_0$, and one of the following must hold:
\begin{equation*}
\begin{aligned}
M &=M' +[FMO_-], \\
M &=M' + [O_+MO_-F],\\
M &=M'+ [O_- M O_+ F].
\end{aligned}
\end{equation*}
In particular, $M$ is obtained from $M'$ by a surgery which attaches a M{\"o}bius band. 
\end{prop}
\begin{proof}

Let us first recall some facts about closed surfaces $P$ with involutions $\rho$~\cite{Natanzon}. The factor $P/\rho$ can be given a structure of a surface with boundary, where the ovals in $\rho$ project to the boundary and fixed points are the branch points of the projection. Suppose that $P/\rho$ has genus $0$ and let $F, C=C_-+C_+$ be as in Section~\ref{S:BRS}, so that $P/\rho$ has exactly $C$ boundary components. Then by the Riemann-Hurwitz formula one has that $F+C = 0$ implies that $P$ is a disjoint union of two spheres with $\rho$ exchanging the copies. Otherwise, $F+C_- = 0$ implies that $(P,\rho)$ is an orientable BRS. Finally $F+C_- = 2$ implies that either $(P,\rho)$ is a non-orientable BRS (if $C>0$) or is a $T_0^\mathrm{sk}$ (if $C=0$), where the latter follows from the Riemann-Hurwitz formula.

Recall that $(M',T')$ is a $T'$-invariant connected component of $(M'_\fC, T'_\fC)$, where $(M'_\fC, T'_\fC)$ is obtained from $(M, T)$ by collapsing a $T$-invariant collection $\fC$ of pairwise disjoint, non-contractible simple closed curves in $M$. Let $\tau$ be the basic reflection on $M$ and let $N:=M/\tau$, $c:=\fC/\tau$.
Since the degeneration is elementary we may assume that $c$ is connected. Let $N_c$ be the surface obtained by collapsing $c\subset N$ and let $M'_\fC\to N_c$ be the corresponding projection. 

We consider several cases depending on the geometry of $c$. 
Note that in all cases $N_c$ is a either a surface of genus $0$ or a disjoint union of two genus $0$ surfaces. We start with the cases where $N_c$ remains connected: it turns that in those cases $M'_\fC = M'$, and we denote by $F',C'=C'_++C'_-$ the number of fixed points and (un)twisted ovals of $M'$.

{\em Case 1.} The curve $c$ connects two branch points, so that $N_c$ is a connected surface of genus $0$. By inspection, $\fC$ is connected and its neighbourhood has to be an $S^{1,1}$-antitube, so that the points of the collapsed set $P_\fC$ are not fixed points for $T'_\fC$. We observe that $M'_\fC$ is connected and therefore coincides with $M'$, is a BRS, and has two fewer fixed points than $M$. Since $F+C_- = 2$, we see that $F' +C_-' =0$, so that $(M',T')$ is an orientable BRS. Since $C>0$, $C_+'=C_+$ we can write
\[
M=M' + [O_+MO_-F] +[O_-MO_+F]
\]
and the theorem is proved in this case.

{\em Case 2.} The curve $c$ is contained in $\partial N$, so that $N_c$ is a connected surface of genus $0$. By inspection, $\fC$ must be one of the following:
\begin{enumerate}
    \item An $S^{1,0}$-antitube. In this case, the points of $P_\fC$ are not fixed points for $T'_\fC$. Once again $M'=M'_\fC$ is a surface with one fewer untwisted oval and $F'+C_-'=2$, so that it is either a non-orientable BRS or a $T_0^\mathrm{sk}$. Finally, since $F'+C_-'=2$, we can write
    \begin{align*}
    M &= M' + [FMO_-] + [O_-MO_+F] 
    \quad \text{or}  \\
    M &= M' + [O_-M O_+F] + [FMO_-],
    \end{align*}
    depending on whether $F'>0$ or $C'_->0$.
    \item An $S^1(M)$. Then the single point in $P_\fC$ is fixed, $C_-' = C_--1$, $F'=F+1$, so that $F'+C_-' = 2$, and
    \[
    M= M'+[FMO-]
    \]
\end{enumerate}

{\em Case 3.} The curve $c$ connects a branch point to $\partial N$. Then $\fC$ is $S^{1,1}(M)$ and collapsing it removes a fixed point and changes the twistedness of an oval, so that $F'=F-1$, $C'=C$ and $(M',T')$ is a BRS, either orientable or not. We can write
\[
M=M'+[O_+MO_-F]\quad\text{ or }\quad M=M'+[O_+MO_-F]
\]
depending on whether $C_-'>0$ or $C_+'>0$

{\em Case 4.} The curve $c$ connects two points on different boundary components of $\bd N$, so that $N_c$ is connected. By inspection, $\fC$ is an $S^{1,1}$-tube, and the points in $P_\fC$ belong to an oval, so that no new fixed points added. Collapsing $\fC$ leads to the merger of two ovals with their twistedness added up: that is, either

\begin{enumerate}
    \item two untwisted ovals merge to an untwisted oval, so that $M'=M'_\fC$ and  $0<C'_+ = C_+-1$, $F=F'$, $C_-'=C_-$. This is analogous to the case 2.(1).
    \item one twisted and one untwisted oval merge to a twisted oval, and again $C'_+ = C_+-1$, $F=F'$, $0<C_-'=C_-$.
    \item two twisted ovals merge to an untwisted oval. Then $C'_-=C_--2$, so that $F'+C_-' = 0$, $C_+'>0$ and $M'$ is an orientable BRS. We can then write
    \[
    M = M' + [O_+M O_-F] + [FMO_-]
    \]
\end{enumerate}

We now turn to cases the where $N_c$ is is disconnected, in which case we still denote by $F',C'=C'_++C'_-$  the invariants of $M'$, but we also use $F'',C''=C''_++C''_-$ for the invariants of the remaining connected components of $M'_\fC$

{\em Case 5.} The curve $c$ is closed and avoids $\bd N$ and branch points. Then $\fC$ is either
\begin{enumerate}
    \item a collection of two curves permuted by $\tau$ or an $S^1_a$-antitube, then $P_\fC$ contains no fixed points of $T'_\fC$, so that $F=F'+F''$, $C_-= C'_-+C_-''$, $C_+ = C_+'+C_+''$. Note that $F''+C''>0$ as otherwise $\fC$ is contractible. If $F'+C_-' = 2$, then $M'$ is a non-orientable BRS or $T_0^{\mathrm{sk}}$, $F''+C_-'' = 0$, hence, $C_+''>0$, which is again analogous to Case 2.(1). If $F'+C_-' = 0$, then $C_+'>0$, $M'$ is an orientable BRS and $C_+''=0$ (as otherwise the degeneration is not elementary), so that we can write
    \[
    M=M'+[O_+MO_-F]+[O_-MO_+F] + C_-[FMO_-].
    \]
    \item or an $S^1_a$-tube.
\end{enumerate}

In the second case both points $P_{\fC}$ are fixed by $T'_\fC$ so that $F+2= F'+F''$, $F',F''>0$ and still $C_-= C'_-+C_-''$, $C_+ = C_+'+C_+''$. Therefore, both components of $M'_\fC$ are non-orientable and $F'+C_-'=F''+C_-'' = 2$. Since the degeneration is elementary, $C''_+=0$. If $C_-'' = 1$, then $F''= 1$, $F' = F+1$, $C_-' = C_--1$ so that 
\[
M = M' +[FMO_-]. 
\]
If $C_-''=0$, then the curve $c$ encloses a disk in $N$ that contains a single branch point. By Riemann-Hurwitz formula, the preimage of this disk in $M$ is a still a disk, so that $\fC$ is contractible. 

{\em Case 6.} The curve $c$ connects two points on the same boundary component. Then $\fC$ is an $S^{1,1}$-tube and it splits an oval of $M$ into two ovals preserving the total twistedness: i.e. a twisted oval splits into twisted and untwisted ovals, or an untwisted oval splits into two twisted or two untwisted ovals. In this case, $C+1=C'+C''$, $C',C''>0$ and $F=F'+F''$. In particular, $F''+C_-''\leq 3$, so $F''+C_-''$ must be $2$ or $0$. Suppose first that $F''+C_-'' =2$, then since the degeneration is elementary, $C_+'' = 0$, so that $C_-''>0$. Then either a twisted oval splits into a twisted and untwisted: i.e. $C_++1 = C_+'$, $C_- = C_-'+C_-''$, so that
\[
M = M'+[O_+MO_-F] + (C_-''-1)[FMO_-];
\]
or an untwisted oval split two twisted ones: i.e. $C_+ = C_+'+1$, $C_-+2 = C_-'+C_-''$, so that
\[
M = M'+[O_-MO_+F] + (C_-''-1)[FMO_-].
\]

Suppose now that $F''+C_-'' =0$. Then $C_+''\geq 1$, and if $C_+''=1$, then $c$ encloses a half-disk whose preimage under the projection is a disk, so that $\fC$ is contractible. Therefore, $C_+''>1$. Then $C_-=C_-'$, $C_+'=C_+-1$, $F'=F$, so this is analogous to Case 2.(1).  

This concludes the proof of the proposition.
\end{proof}

\begin{remark}
    \label{rmk:RP2_initial}
    Note that if the genus of $M$ is at least $2$, then either $C_->0$ or $C_+>0, F>0$ so that one can find a non-orientable BRS $(M',T')$ such that
    \[
    M = M' + [FMO_-]\quad\text{ or }\quad M = M' + [O_-MFO_+].
    \]
\end{remark}

\begin{prop}
\label{PZ2Sph}
Let $(M, \tau)$ be a spherical basic reflection surface and $T$ the corresponding $\Z_2$-action.  Then there is a smooth $T$-invariant metric on $M$ realizing $\Lambda_1^T(M)$, induced by a minimal embedding $M \rightarrow \Sph^4$ by $\lambda_1$-eigenfunctions which satisfies the conclusions of Lemma \ref{Lgraphsph}.
\end{prop}
\begin{proof}
Let $\gamma = 1$, in which case $M = \RP^2$. By work of Li-Yau \cite{LiYau}, the round metric on $\RP^2$ is the unique metric maximizing $\bar\lambda_1$ and it admits a basic reflection with $C_-=F=1$.  By Theorem \ref{Tclass}, this is the unique basic reflection surface structure on $M$, and as shown in \cite{LiYau}, in this case $\Lambda_1^T(M) = 12\pi$.

Let $(M, \tau)$ be a basic reflection surface with $F+C_- = 2$ and nonorientable genus $\gamma\geq 2$.  In this case, \cite[Theorem 7.7]{KKMS} asserts that 
\begin{equation}
\label{ineq:MM'2}
\Lambda_1^T(M)\geq\max_{(M',T')\prec (M,T)} \left\{\Lambda_1^{T'}(M'), 12 \pi \right\},
\end{equation}
and furthermore that strictness in the inequality ~\eqref{ineq:MM'2} implies the existence of a maximal metric achieving $\Lambda_1^T(M)$. As soon as the strict inequality holds, Lemma~\ref{Lgraphsph} implies that the maximal metric is in fact smooth everywhere. We will show by induction that inequality~\eqref{ineq:MM'2} is strict for $\gamma\geq 2$.

Let $(M', \tau') \prec (M, \tau)$ be an elementary degeneration.  By Proposition \ref{Pelem0}, $(M',T')$ is a basic reflection surface, and $(M, T)$ can be obtained from $(M', T')$ by an equivariant surgery attaching a single M{\"o}bius band.
Combining the step of induction (for non-orientable BRS) with~\cite[Theorem 8.7]{KKMS} (for orientable BRS) there exists a {\em smooth} metric on $M'$ realizing $\Lambda_1^{T'}(M')$. In particular, a straightforward adaptation of the arguments in \cite[Section 3]{KSP} to the equivariant case (as is explained in the proof of Lemma~\ref{twist.bd}) shows that $\Lambda^T_1(M) > \Lambda_1^{T'}( M ')$. It remains to show $\Lambda_1^T(M)>12\pi$. It follows from Remark~\ref{rmk:RP2_initial} that it is possible to find a chain of degenerations starting at $(M,T)$ and terminating at $\mathbb{RP}^2$. In particular, $\mathbb{RP}^2\prec(M,T)$, which together with the base of induction and previous arguments implies that $\Lambda_1^T(M)>\Lambda_1(\mathbb{RP}^2) = 12\pi$. 
\end{proof}

We now turn to the elliptic basic reflection surfaces.

\begin{theorem}
\label{Telliptic}
Let $(M, T)$ be an elliptic basic reflection surface.  Then 
\begin{align*}
8\pi^2/\sqrt{3} \leq  \Lambda_1^T(M) < 16\pi
\end{align*}
and there is a smooth $T$-invariant metric on $M$ realizing $\lambda^T_1(M)$, induced by a minimal embedding $M \rightarrow \Sph^6$ by $\lambda_1$-eigenfunctions which satisfies the conclusions of Lemma \ref{Lgraphsph}.
\end{theorem}
\begin{proof}
 A straightforward modification of the proof of Proposition \ref{Pelem0} shows that if 
\begin{align*}
(M', T') \prec (M, T)
\end{align*}
is an elementary degeneration, then $(M', T')$ is either a basic reflection surface or the space $T^{\mathrm{sk}}_1$  defined in \ref{sssurg}, and $M$ is obtained from $M'$ by a surgery which attaches a M{\"o}bius band.

Now recall from \cite{Nadirashvili} that $\Lambda_1(\mathbb{T}^2) = 8\pi^2/\sqrt{3}$, where $\mathbb{T}^2$ is the torus, and that equality is achieved by the flat metric on the equilateral torus.  It is easy to see there is a hyperelliptic involution of $T^2$ which is an isometry for this metric, so that $\Lambda_1(\mathbb{T}^2) = \Lambda_1^T(T^{\mathrm{sk}}_1)$.

Since there is a topological degeneration of $(M, T)$ to $T^{\mathrm{sk}}_1$ obtained by collapsing all the ovals on $M$, the inequality $\Lambda^T_1(M) \geq 8\pi^2/\sqrt{3}$ follows from the preceding and \eqref{ineq:MM'}.

The existence of a smooth, $T$-invariant maximal metric $g$ on $(M, T)$ now follows from an inductive argument similar to the one in the proof of Proposition \ref{PZ2Sph}, and we omit the details.  By Corollary \ref{Prealcl}, there is then an isometric minimal embedding $\iota: (M, g) \rightarrow \Sph^n$ by first eigenfunctions satisfying the conclusions of Lemma \ref{Lgraphsph}.  Finally, the bound $n \leq 6$ follows from the multiplicity bound from Proposition \ref{PmultBRS}.
\end{proof}

We now consider the hyperelliptic case and show, as remarked earlier, that hyperelliptic basic reflection surfaces \emph{do not} have $\Lambda^T_1(M)$-maximizing metrics, where $T$ is the $\Z_2$-action. 
 
\begin{theorem}
\label{Tmainhigh}
Let $(M, T)$ be a hyperelliptic nonorientable basic reflection surface.  Then
\begin{align}
    \label{Emax}
\Lambda_1^T(M) = 16\pi
\end{align} 
and there is no $T$-invariant metric on $M$ realizing $\Lambda^T_1(M)$.
\end{theorem}
\begin{proof}
We first prove \eqref{Emax}.
Let $(M', T')$ be the space $T^{\mathrm{sk}}_g$ equipped with the $\Z_2$-action fixing $2g+2$ points as described in Section \ref{sssurg}.  By collapsing all ovals  we have that
\begin{align*}
(M', T') \prec (M, T).
\end{align*}
Next, denote by $(M'', T'')$ the $\Z_2$-space $T^{\mathrm{sk}}_2$.  Since $g \geq 2$,
there is a further topological degeneration 
\[(M'', T'') \preccurlyeq (M', T')\]
obtained by equivariantly collapsing $g-2$ handles on $M' = T^{\mathrm{sk}}_g$.  (If $g = 2$, we of course have that $(M'', T'') = (M', T')$).

Denoting by $M_2$ the genus-$2$ orientable surface, we now recall from \cite{NaySh} that $\Lambda_1(M_2) = 16\pi$, and that $\Lambda_1(M_2)$ is realized by a singular metric induced from a branched cover $M_2 \rightarrow S^2$ with $6$ ramification points.  In particular, this implies
\begin{align*}
\Lambda_1^{T''}(M'') = \Lambda_1(M_2) = 16\pi.
\end{align*}
Putting together the preceding and combining with \eqref{ineq:MM'} implies \eqref{Emax}.

Next, recall from \cite[Lemma 5.19(ii)]{KKMS} that $\bar{\lambda}_1(M, g) < 16\pi$ for any $T'$-invariant smooth metric on $M$.  Combined with \eqref{Emax}, this proves that $\Lambda_1^{T}(M)$ cannot be achieved by a smooth metric.
\end{proof}

\begin{remark}
    One can say even more: there does not exist a $T$-invariant conformal class $\mathcal C$ on $M$ such that $\Lambda_1^T(M,\mathcal C) = 16\pi$. This follows from the existence of maximizers in conformal classes combined with the fact that the proof of~\cite[Lemma 5.19(ii)]{KKMS} follows through for metrics with conical singularities.
\end{remark}

\begin{remark}
    \label{rmk:no_max}
    Let us explain where the approach of~\cite{KSP} breaks down for hyperelliptic surfaces. The argument is~\cite[Section 3]{KSP} proves that it is possible to attach a cross-cap at a {\em smooth} point of a $\bar\lambda_1(M',T')$-maximal metric to strictly increase the first normalized eigenvalue. If $(M',T') = T_2^\mathrm{sk}$, then, first, any $\bar\lambda_1(T_2^\mathrm{sk})$ has conical singularities at the fixed-points of the skewer involution and, second, there are topological degenerations that specifically require attaching a cross cap at a fixed point. The combination of these two facts means that the construction of~\cite[Section 3]{KSP} cannot be applied. This example shows that the geometric conditions for the existence of maximizers such as~\cite[Section 3]{KSP},~\cite[Theorem 8.1]{KKMS} or~\cite{PetridesMax} are likely unavoidable and are not artifacts of the methods. 
\end{remark}

\subsection{The families $\zeta_k^i$}
\label{ssSph2main}

In this section we state the existence result for the families $\zeta_k^i$. Its proof is analogous to the proof of the Proposition~\ref{PZ2Sph}. 
\begin{theorem}
\label{TsphDk}
Let $M=\zeta_k^i$ and $T$ be one of the actions described in Example~\ref{ex:zetaki}. Then 
\[
\Lambda_1^T(M)<16\pi
\]
and there exists a smooth $T$-invariant metric on $M$ realizing $\lambda_1^T(M)$, induced by a minimal embedding $M\to \Sph^4$ by $\lambda_1$-eigenfunctions which satisfies the conclusions of Lemma~\ref{Lgraphsph}
\end{theorem}

\begin{proof}[Sketch of the proof]
    The proof follows the proof of Proposition~\ref{PZ2Sph}. Assume that $i=0$, then the factor $M/T$ is a quadrilateral with sides corresponding to fixed-point sets of generating reflections. For elementary degenerations, the factor $c$ of the set $\fC$ is a segment in this quadrilateral. Considering different cases for $c$, we observe that there are no $T_\fC'$-invariant connected components of $(M_\fC',T_\fC')$, which clearly leads to the existence.

    We leave details and the cases $i=1,2$ to the interested reader.
\end{proof}

\appendix
\section{Facts about basic reflection surfaces}
\label{ssdoub}

\subsection{Minimal immersions to spheres}
In \cite[Section 5]{KKMS}, we studied minimal immersions of basic reflection surfaces $(M, \tau)$ in the $n$-sphere $\Sph^n$ by first eigenfunctions.  Such immersions are necessarily embeddings and are in fact minimal \emph{doublings} of a $2$-sphere $\Sph^2 \subset \Sph^n$ in the sense of the following definition (see \cite[Definition 5.34]{KKMS} and \cite{LDG}).

\begin{definition}
\label{ddoub}
Given a Riemannian manifold $N$ and a surface $\Sigma \subset N$, a (surface) \emph{doubling} of $\Sigma$ in $N$ is a smooth surface $M$ in $N$ satisfying
\begin{enumerate}[label={(\roman*)}]
\item The nearest-point projection $\pi : M \rightarrow \Sigma$ is well-defined;
\item $\pi(M)$ is a disjoint union  $\Sigma_1 \cup \Sigma_2$, with $\Sigma_1$ a disjoint union of curves and isolated points, $\Sigma_2\subset \Sigma$  a domain, and moreover
\item $\pi|_{\pi^{-1}(\Sigma_1)} : \pi^{-1}(\Sigma_1) \rightarrow \Sigma_1$ is a diffeomorphism; and
\item $\pi|_{\pi^{-1}(\Sigma_2)} : \pi^{-1}(\Sigma_2) \rightarrow \Sigma_2$ is a smooth $2$-sheeted covering map.
\end{enumerate}
The doubling is called \emph{minimal} if $M$ is minimal, and called \emph{symmetric} if $\Sigma$ is fixed pointwise and $M$ is fixed setwise by an involutive isometry of $N$.
\end{definition}

We record the main conclusions from \cite{KKMS} as follows.

\begin{lemma}
\label{Lgraphsph}
If $\iota: M \rightarrow \Sph^n$ is a full branched minimal immersion with
\begin{enumerate}[label=\emph{(\alph*)}]
\item $(M, \tau)$ is a closed basic reflection surface with genus $>0$, and
\item $\iota$ is a (branched) isometric immersion by first eigenfunctions,
\end{enumerate}
then the following hold:
\begin{enumerate}[label=\emph{(\roman*)}]
\item $n \geq 3$, with $n = 3$ if and only if $M$ is orientable; 
\item $n = 4$, if $M$ is nonorientable and $M \setminus M^\tau$ has genus zero.
\item $\tau = \iota^* \Rcapunder_P$ for some $3$-dimensional subspace $P \subset \R^{n+1}$, where $\Rcapunder_P$ is the isometry of $\R^{n+1}$ with $\Rcapunder_P|_P = \Id_P$ and $\Rcapunder_P|_{P^\perp} = - \Id_{P^\perp}$; 
\item For each oval $O \subset M^\tau$, $\iota(O) \subset \Sph(P)$ is embedded and strictly convex, where $\Sph(P) = \Sph^n \cap P$.
\item $M$ is a symmetric minimal doubling of $\Sph(P)$ in  $\Sph^n$ as in \ref{ddoub};
\item $M$ has area strictly less than $8\pi$.
\end{enumerate}
In particular, $\iota$ is an embedding. 
\end{lemma}
\begin{proof}
Except for item (ii), this is \cite[Lemma 5.35]{KKMS}.  Item (ii) also follows directly from statements in \cite{KKMS}: if $M$ is nonorientable and $M\setminus M^\tau$ has genus zero, then Propositions 5.12 and 5.15 from \cite{KKMS} together imply the dimension of the first eigenspace is at most five, implying $n\leq 4$, since $\iota$ is a full immersion. On the other hand, $n \geq 4$ follows from item (i).
\end{proof}

\begin{cor}
\label{Prealcl}
Let $(M, \tau)$ be a closed basic reflection surface.  If there exists a $\tau$-invariant $\bar{\lambda}_1^\tau$-maximal metric $g$ on $M$, then there exists an isometric minimal embedding $\iota: (M,g) \rightarrow \Sph^n$ by first eigenfunctions satisfying the conclusions of Lemma \ref{Lgraphsph}.
\end{cor}
\begin{proof}
Theorem 2.12 in \cite{KKMS} ensures the existence of an isometric branched minimal immersion $\iota: M \rightarrow \Sph^n$ by first eigenfunctions, and the remaining conclusions follow from Lemma \ref{Lgraphsph}.
\end{proof}

\subsection{Multiplicity bounds for the first eigenvalue}
Let $(M, \tau)$ be a basic reflection surface, and let $\Ecal$ denote the first eigenspace of the Laplacian on $M$.  Then $\tau$ induces a linear involutive isometry $\tau^* : \Ecal \rightarrow \Ecal$ given by $\tau^* u = u \circ \tau$, which induces an orthogonal direct sum decomposition $\Ecal = \Ecal^+\oplus \Ecal^-$ into subspaces of even ($\Ecal^+$) and odd ($\Ecal^-$) functions with respect to $\tau$.  

In \cite{KKMS} we obtained bounds on $\dim \Ecal$ on basic reflection surfaces $(M, \tau)$ where $M \setminus M^\tau$ had genus zero.  The next proposition provides an analogous bound when $M \setminus M^\tau$ has genus one.

\begin{prop}
\label{PmultBRS}
Let $(M, \tau)$ be a basic reflection surface. Then
\begin{enumerate}[label=\emph{(\roman*)}]
\item $\dim \Ecal^+ \leq 3$. 
\item $\dim \Ecal^- \leq 4$, if $M \setminus M^\tau$ has genus one.
\end{enumerate}
In particular, $\dim \Ecal \leq 7$ if $M \setminus M^\tau$ has genus one. 
\end{prop}
\begin{proof} 
Item (i) was proved in \cite[Proposition 5.15]{KKMS}, so it suffices to prove (ii).
Let $(M, \tau)$ be as above.  By Remark \ref{Rbas}, the complement $N \setminus N^\tau$ of the fixed-point set $N^\tau$ in each orientable basic reflection surface $(N, \tau)$ has genus zero.  Since $M \setminus M^\tau$ has positive genus, then $M$ must be nonorientable.

Let $E$ be the set of ends of $M \setminus M^\tau$, and note that $M$ has one end corresponding to each twisted oval, one end corresponding to each isolated fixed-point, and two ends corresponding to each untwisted oval. 
Let $M^\circ : = (M \setminus M^\tau) \cup E$ be the space obtained from $M \setminus M^\tau$ by compactifying each end of $M \setminus M^\tau$ by adding a point, and let $\iota: M\setminus M^\tau \rightarrow M^\circ$ be the inclusion map.  Since $M \setminus M^\tau$ has genus 1, then $M^\circ$ is a torus. 

Next, let $u \in \Ecal^-$, and let $\Ncal_u$ be its nodal set.  Then  $\Ncal_u$ consists (see \cite[Theorem 2.5]{Cheng}) of a finite number of immersed loops, which meet at finitely many points; at each such point, an even number of nodal arcs emanate.  Furthermore, each fixed-point oval $O\subset M^\tau$ is contained in $\Ncal_u$, and since $u$ changes sign when crossing each nodal arc, it follows that an even number of nodal arcs emanate from each such oval $O$.

Now let $\Ncal^\circ_u : = E \cup \iota (\Ncal_u \setminus M^\tau)$.  From the preceding remarks, it is easy to see that $\Ncal^\circ_u$ has the structure of a nodal graph in $M^\circ$; in particular, an even number of arcs in $\Ncal^\circ_u$ emanate from each point in $E\subset \Ncal^\circ_u \subset M^\circ$.

The rest of the proof is split into two cases, based on the number $F$ of isolated fixed-points for $\tau$.

\emph{Case 1:} $F \geq 1$.  Part of the proof in this case is based on ideas in the proof of \cite[Theorem 2.1]{Besson}, which proves that the multiplicity of the first eigenvalue $\lambda_1$ on a genus-$g$ Riemannian surface is bounded by $4g+3$. 

Fix an isolated fixed-point $p \in M^\tau$, and let $u \in \Ecal^-$ be an odd first eigenfunction.  Since $\Ncal^\circ_u \subset M^\circ$ has the structure of a nodal graph and $M^\circ$ is a torus, it follows from the topological argument in \cite[Theorem 3.2]{Cheng} that  the order of vanishing of $u$ at $p$ is at most $3$. 

Since $p \in M^\tau$ and $u$ is odd, it follows that $u(p) = 0$ and $\nabla^2 u(p) = 0$.  If also $\nabla u(p) = 0$, differentiating the linear equation $\mathrm{tr} \nabla^2 u = \Delta u = \lambda_1 u$ and evaluating at $p$ shows each trace of $\nabla^3 u$ vanishes at $p$.  Thus, the system
\begin{align*}
u(p) = 0, 
\quad
\nabla u(p) = 0, 
\quad
\nabla^2 u(p) = 0, 
\quad
\nabla^3 u(p)=0
\end{align*}
amounts to a homogeneous system of four linear equations.  It follows that $\dim \Ecal^- \leq 4$, for if not, some nonzero $u \in \Ecal^-$ would vanish to fourth order.

\emph{Case 2:} $F = 0$.  By Theorem \ref{Tclass} and the assumption that $M \setminus M^\tau$ has genus one, it follows that $F+C_- = 4$, so $M$ has four twisted ovals.  Let $O$ be one of them, fix distinct points $p_1, p_2, p_3, p_4 \in O$ and consider the linear map
\begin{align*}
&T: \Ecal^- \rightarrow T^*_{p_1}M \times T^*_{p_2} M \times T^*_{p_3} M \times T^*_{p_4} M, 
\\
& Tu = (du_{p_1}, du_{p_2}, du_{p_3}, du_{p_4}).
\end{align*}
For any $u \in \Ecal^-$, $u$ vanishes on $O$, so the image $\mathrm{Im}\, T$ lies in a $4$-dimensional subspace of $T^*_{p_1}M \times T^*_{p_2} M \times T^*_{p_3} M \times T^*_{p_4} M$. 

If it were the case that $\dim \Ecal^- \geq 5$, then there would be a nonzero $u \in \Ecal^-$ in the kernel of $T$, and then $u(p_i)= 0$ and $du_{p_i} = 0$.  Consequently, there is a nodal line in $\Ncal_u$ meeting $O$ transversely at $p_i$.

Since $O$ is twisted, then $M \setminus O$ has one end. Let $e \in E$ be the end of $M \setminus M^\tau$ corresponding to $O$.  By the earlier discussion and the paragraph above, it follows that at least four lines in $\Ncal^\circ_u$ cross at $e \in M^\circ$.  Since $M^\circ$ is a torus and $\Ncal^\circ_u$ has the structure of a nodal graph, this contradicts the topological argument in the proof of \cite[Theorem 3.2]{Cheng}, which asserts that at most three lines in $\Ncal^\circ_u$ may meet at $e$.  
\end{proof}

The next proposition establishes bounds on $\dim \Ecal$ on $D_2$-invariant basic reflection surfaces satisfying the conditions in \ref{ssD2}. 

\begin{proposition}\label{mult.prop}
Let $T: \Gamma \times M \rightarrow M$ be an action of the group $\Gamma = D_2$ on a nonorientable closed surface $M$ satisfying the conditions in \ref{ssD2}.  Then $\mult\lambda_1(M,T)\leq 5$.
\end{proposition}
\begin{proof}
Let $\rho_1, \rho_2$ be generators for $D_2 = \Z_2 \times \Z_2$ as in \ref{ssD2}.  
Let us decompose the first eigenspace $\Ecal$ according to the representations of $D_2$, so
\begin{align*}
\Ecal = \Ecal^{+}_+ \oplus \Ecal^{+}_- \oplus \Ecal^{-}_+ \oplus \Ecal^-_-, 
\end{align*}
where the upper and lower $\pm$ indicate that each function in $\Ecal^\pm_\pm$ is even ($+$) or odd ($-$) with respect to $\rho_1$, and $\rho_2$, in that order.  Any nonzero element $u$ of an $\Ecal^\pm_\pm$ then satisfies the following mixed Dirichlet-Neumann conditions on the boundary $\partial C$ of the chamber $C$, where $\tau:= \rho_1\rho_2$: 
    \begin{itemize}
        \item $u \in \Ecal^+_+$ has Neumann boundary conditions on $\partial C$; 
        \item $u \in \Ecal^+_-$ has Dirichlet conditions on the edges labeled $\rho_2$ and $\tau$, and Neumann conditions elsewhere;
        \item $u \in \Ecal^-_+$ has Dirichlet conditions on the edges labeled $\rho_1$ and $\tau$, and Neumann conditions elsewhere;
        \item $u \in \Ecal^-_-$ has Dirichlet conditions on the edges labeled $\rho_1$ and $\rho_2$, and Neumann conditions elsewhere.
    \end{itemize}
    
In order to prove the proposition, it suffices to prove that
\begin{align}
\label{EineqD2}
\dim \Ecal^+_+ \leq 2, 
\quad
\dim \Ecal^+_- \leq 1, 
\quad
\dim \Ecal^-_+ \leq 1, 
\quad
\dim \Ecal^-_- \leq 1. 
\end{align}

The last three inequalities in \eqref{EineqD2} are clear, since the first eigenvalues $\lambda^{ND}_1(C), \lambda^{DN}_1(C), \lambda^{DD}_1(C)$ of the corresponding mixed problem are simple.  For $\Ecal^+_+$ the situation is more complicated as the constant functions are first eigenfunctions. 

    
If it is not the case that $\dim \Ecal^+_+\leq 2$, there exists an eigenfunction $u$ with vanishing order at least $2$ at an outer boundary point $p \in \partial C$. Hence, there exist two nodal lines that start at $p$. If at least one of those lines ends on the outer boundary component, then the $\Gamma$-translates of the two regions between that line and the outer boundary component form a nodal domain, so these two regions are the nodal domains of $u$, contradicting the fact that there is another nodal line of $u$. 
    
    Let us now ``collapse'' all the interior loops to form a reduced nodal graph on the disk, as for example in~\cite[Section 3]{Kokarev}. In the reduced nodal graph, all interior loops become new vertices, which have even degree as $u$ always changes sign across a nodal line. It is then easy to see (or prove using the Euler inequality~\cite{Kokarev}) that the only way for the reduced graph to have two nodal domains is to form a loop starting and ending at $p$. Consider the domain enclosed by this loop. 
     Its preimage is a nodal domain in $C$ whose boundary contains only the label $\tau$; therefore, the $\Gamma$-translates of that nodal domain is a collection of two nodal domains, so that the corresponding $\Gamma$-invariant function on $M$ has at least $3$ nodal domains.  This contradicts the Courant nodal domain theorem. 
\end{proof}

\section{A result of Natanzon}
\label{SNat}

Here we prove Theorem \ref{thm:Nat2}, which we restate here as Theorem \ref{thm:Nat2app} below.  Though the proof of this result is essentially contained in work of Natanzon \cite{Natanzon}, we present it here for completeness and convenience of the reader.   

\begin{theorem}
    \label{thm:Nat2app}
    Let $\alpha$, $\beta$ be two non-commuting reflections on the orientable surface of genus $\gamma>1$ which generate a finite group. Then the total number of ovals of $\alpha$ and $\beta$ is at most $\gamma+4$.
\end{theorem}

\begin{remark}
    In fact, from the arguments in~\cite{Natanzon} one can extract a sharper bound of $\gamma+3$.
\end{remark}

Suppose $\alpha, \beta, \gamma, M$ are as in Theorem \ref{thm:Nat2app}, 
let $G = \la \alpha,\beta \ra$, and let $G^+\leq G$ be the subgroup of orientation preserving transformations. Let us first record some elementary facts. It is well-known that the factor $M/G_+$ can be endowed with the structure of a Riemann surface, such that the projection $\pi\colon M\to M/G_+$ is a branched cover. Applying the Riemann-Hurwitz formula yields
\begin{equation}
    \label{eq:RH}
    \gamma-1 = |G_+|(\gamma'-1) + \frac{1}{2}\sum_{p\in M}\left( |(G_+)_p|-1\right),
\end{equation}
where $\gamma'$ is the genus of $M/G_+$.

Recall also Harnack's theorem (for example \cite[Theorem 5.13]{Dugger}) which states that a reflection on a surface of genus $\gamma$ has at most $\gamma+1$ ovals.

We define a set $B\subset M$ as follows: $p\in B$ iff $p$ is a fixed-point of an involution in $G_+$ that lies on an oval of a reflection in $G$. The next lemma states that $B$ is a branching set for the restriction of $\pi$ to the ovals of $G$.

\begin{lemma}[Lemma 2.3 in~\cite{Natanzon}]
    For any oval $a$ of a reflection in $G$, the restriction $\pi|_{a\setminus B}$ is locally bijective.
\end{lemma}
\begin{proof}
    Suppose that $p\in a$ and there exists a sequence $a\ni q_i,q_i'\to p$ such that $q_i = g_iq_i'$ for some $g_i\in G_+$. Then, up to a choice of a subsequence, $q_i = gq_i'$ for the same $g\ne e$. Passing to the limit, we obtain that $gp=p$ and that the tangent vector to $a$ at $p$ is an eigenvector of $d_pg$. The only non-trivial orientation-preserving symmetry of finite order on $2$-dimensional space is $-id$, which yields that $g$ is an involution, so that $p\in B$ by definition.
\end{proof}

Let us introduce the following notation. We write $\|\alpha\|$ for the number of ovals of the reflection $\alpha$ and $\|G\|$ for the total number of ovals of reflections in $G$. Furthermore, we write $\|G_B\|$ for the total number of ovals of reflections in $G$ that do not intersect $B$.

Recall the projection $\pi\colon M\to M/G_+$. Since $G_+\triangleleft G $, all reflections in $G$ descend to a single reflection $\rho$ such that $\pi\alpha = \rho \pi$. Therefore, if $a$ is an oval of a reflection in $G$, then $\pi(a)$ is fixed by $\rho$, hence, it is contained in an oval of $\rho$. Furthermore, if $a\cap B=\varnothing$, then $\pi|_a$ is a cover, so that $\pi(a)$ is an oval of $\rho$.  

For later use, we note that combining \eqref{eq:RH} with Harnack's theorem applied to the reflection $\rho$ on $M/G_+$ implies
\begin{align}
\label{ERH2}
\gamma-1 \geq |G_+|(\| \rho\|-2) + \frac{1}{2} \sum_{p \in M} ( |(G_+)_p|-1).
\end{align}

\begin{lemma}
    \label{lem:Bempty}
    One  has  $\|G_B\|\leq |G_+|\|\rho\|$. In particular, if 
     $B=\varnothing$, then $\|G\|\leq |G_+|\|\rho\|$.
\end{lemma}
\begin{proof}
    Each oval in $G_B$ descends to an oval of $\rho$; since the degree of $\pi$ is $|G_+|$, at most $|G_+|$ ovals descend to the same oval of $\rho$.

    If $B=\varnothing$, then $\|G_B\| = \|G\|$.
\end{proof}

\begin{lemma}
    \label{lem:dihed}
    Let $\Gamma$ be a finite group generated by two involutions $\gamma_1\ne\gamma_2$. Then $\Gamma\cong D_l$ for some $l>2$. If $\gamma_1$ is not conjugate to $\gamma_2$, then $l=2m$ for some $m\in \mathbb{N}$.
\end{lemma}
\begin{proof}
    Since $\Gamma$ is finite, there is some $l\in \mathbb{N}$ such that $(\gamma_1\gamma_2)^l = e$, so that $\Gamma$ is a quotient of the dihedral group $D_l$ by a normal subgroup $N$. Thus, there are only the following possibilities for $N$:
    \begin{enumerate}
        \item $N = \Z_l$, the group of rotations. Then $\Gamma = \Z_2$ and $\gamma_1=\gamma_2$, a contradiction;
        \item\label{cas:only_pos} $N$ is trivial, then $\gamma_1$ is conjugate to $\gamma_2$ unless $l$ is even;
        \item If $l=2m$ is even, then there are two normal subgroups $N_1,N_2$ isomorphic to $D_m$ (symmetries of two $m$-gons inscribed in the $2m$-gon). But $N_1$ contains one of the reflection generators of $D_l$ while $N_2$ contains the other, so one of $\gamma_i = e$ in $\Gamma$.
    \end{enumerate}
    Thus, item~\eqref{cas:only_pos} is the only possibility.
\end{proof}

\begin{proposition}
    If $B=\varnothing$, then Theorem~\ref{thm:Nat2} holds.
\end{proposition}
\begin{proof}
    Using \eqref{ERH2} and Lemma \ref{lem:Bempty}, we obtain
    \[
    \gamma-1 \geq |G_+|(\|\rho\|-2) \geq \|G\| - 2|G_+| \geq \|G\|-|G|,
    \]
    where we used that $|G| \geq 2 |G_+|$ in the last step.
    By Lemma~\ref{lem:dihed}, $|G|=2l$ and $\|G\| = \frac{l}{2}(\|\alpha\| + \|\beta\|)$. As a result,
    \[
    \|\alpha\| + \|\beta\|\leq \frac{2}{l}(\gamma-1) + 4\leq \gamma+3.
    \]
    Note that in this case we have not used that $\alpha,\beta$ do not commute.
\end{proof}

\begin{lemma}
    \label{lem:GGB}
    One has
    \[
    \|G\| - \|G_B\| \leq \frac{1}{2}\sum_{p\in B} |(G_+)_p|. 
    \]
\end{lemma}
\begin{proof}
    Let $A$ be the union of all ovals in $G$ containing a point from $B$. Let us view $A$ as a graph on $M$ with $B$ being the vertices. Then the degree of each vertex is exactly $2|(G_+)_p|$ and observe that there are no loops in this graph: by the collar lemma, if an involution fixes an oval, then it fixes at least two points on it. Thus, for this graph, $E\geq 2(\|G\| - \|G_B\|)$ and $2E=\sum_{p\in B}2|(G_+)_p|$. Combining these inequalities completes the proof.
\end{proof}

\begin{lemma}
    \label{lem:Gdiv4}
    If $B\ne\varnothing$, then $|G|$ is divisible by $4$ and
    \[
    \|G\| \leq 2\gamma+ |G|.
    \]
\end{lemma}
\begin{proof}
    By Lemma~\ref{lem:dihed}, $G\cong D_l$. If $l$ is odd, then $G_+$ does not contain any involutions, so $B$ is empty; therefore, $l$ is even and $|G| = 2l$ is divisible by $4$. 
    
    Now let $\beta\in G_+$ be an involution. Then by applying the Riemann-Hurwitz formula to the projection $M\to M/\la \beta\ra$, one obtains
    \[
    \gamma-1\geq \frac{1}{2}|\Fix(\beta)| - 2. 
    \]
    Since $|B|\leq |\Fix(\beta)|$, we obtain $|B|\leq 2\gamma+2$.

    By estimating the right hand side of \eqref{ERH2} using Lemmas~\ref{lem:Bempty} and~\ref{lem:GGB} and the preceding, one obtains
    \[
    \gamma-1 \geq \|G_B\| - |G| + \|G\|-\|G_B\| - \frac{1}{2}|B|\geq \|G\| - |G| - (\gamma+1).
    \]
    Rearranging proves the claim.
\end{proof}

\begin{proposition}
    If $B\ne\varnothing$, then Theorem~\ref{thm:Nat2} holds.
\end{proposition}
\begin{proof}
    By Lemma~\ref{lem:Gdiv4} we have $|G| = 4m$ and since $G$ is not commutative we also have $m>1$. By Lemma~\ref{lem:dihed} one has $\|G\| = m(\|\alpha\|+\|\beta\|)$. Combining with Lemma~\ref{lem:Gdiv4} one concludes
    \[
    \|\alpha\|+\|\beta\|\leq \frac{2}{m}\gamma + 4\leq \gamma+4.
    \]
\end{proof}

\bibliography{bibliography}
\end{document}